\documentclass[final]{IEEEtran}  
\PassOptionsToPackage{appendix=true}{clevethm}
\usepackage[thmreset=section, noload={caption,subcaption,algorithmic}]{myPreamble} 
\crefname{appsec}{Appendix}{Appendices}  

\showgrayfalse 

\let\oldCref\Cref\relax
\renewcommand\Cref[1]{%
	\IfStrEq{#1}{Alg:dist}{%
		\hyperref[Alg:dist]{TriPD-Dist~(Alg.~3)}%
	}{%
		\IfStrEq{#1}{Alg:synch-1}{%
		\hyperref[Alg:synch-1]{TriPD~(Alg.~1)}%
	}{%
		\IfStrEq{#1}{Alg:BC}{%
		\hyperref[Alg:BC]{TriPD-BC~(Alg.~2)}%
	}{%
		\oldCref{#1}%
	}%
	}%
	}%
}

\preto\subequations{\ifhmode\unskip\fi} 

\usepackage[noadjust]{cite}  

\usepackage{mathrsfs}
\usepackage{scalerel}
\newcommand\oblong{\mathbin{\scaleobj{0.75}{\square}}}

\newcommand{\y}{\mathsf{y}}
\newcommand{\w}{\mathsf{w}}

\newcommand{\h}{\mathsf{h}}

\DeclareMathOperator\blkdiag{blkdiag}
\DeclareMathOperator\ri{ri}

\DeclareMathOperator\gra{gra}

\let\nabla\OLDnabla

\let\id\relax
\newcommand\id{{\rm Id}}

\DeclareMathOperator*{\E}{\mathbb{E}}

\newcommand{\Nn}{{\rm{I\!N}}}

\usepackage{booktabs,tabularx}

\def\captionsetup#1{}
\makeatletter
	\newcounter{ALC@unique}
\makeatother
\usepackage{algorithmicx,algpseudocode}

	\definecolor{myblue}{rgb}{0, 0.1, 1}
\def\hl#1{{\leavevmode\color{myblue}#1}}	

\newcommand{\hide}[1]{}

\definecolor{mygreen}{rgb}{0.0, 0.5, 0.0}
\def\nf#1{{\color{red}#1}}

\newcommand{\cf}{\textit{cf. }}

\AtBeginEnvironment{rem}{\let\qedsymbol\relax}

\begin{document}
\title{\LARGE \bf
A New Randomized Block-Coordinate Primal-Dual Proximal Algorithm for Distributed Optimization}

\author{Puya Latafat, Nikolaos M. Freris, Panagiotis Patrinos 
	\thanks{Puya Latafat\textsuperscript{1,2} \ puya.latafat@kuleuven.be
		
		Nikolaos M. Freris \textsuperscript{3} \
		nfreris2@gmail.com
		
		Panagiotis Patrinos\textsuperscript{1} \ panos.patrinos@esat.kuleuven.be}
	\\
	\thanks{ 
		The work of the first and third authors was supported by: FWO PhD grant 1196818N; FWO research projects G086518N and G086318N; KU Leuven internal funding StG/15/043; Fonds de la Recherche Scientifique -- FNRS and the Fonds Wetenschappelijk Onderzoek -- Vlaanderen under EOS Project no 30468160.

		The work of the second author, while with New York University Abu Dhabi and New York University Tandon School of Engineering, was supported by the US National Science Foundation under grant CCF-1717207. 

			\textsuperscript{1}KU Leuven, Department of Electrical Engineering (ESAT-STADIUS), Kasteelpark Arenberg 10, 3001 Leuven-Heverlee, Belgium.
			
		\textsuperscript{2}IMT School for Advanced Studies Lucca, Piazza San Francesco 19, 55100 Lucca, Italy.
		
		\textsuperscript{3}
		University of Science and Technology of China, School of Computer Science and Technology, Hefei, 230000, China.
	} 
}

\maketitle

\begin{abstract}	
This paper proposes TriPD, a new  primal-dual algorithm for minimizing the sum of a Lipschitz-differentiable convex function and two possibly nonsmooth convex functions, one of which is composed with a linear mapping. We devise a randomized block-coordinate version of the algorithm which converges under the same stepsize conditions as the full algorithm. It is shown that both the original as well as the block-coordinate scheme feature linear convergence rate when the functions involved are either piecewise linear-quadratic, or when they satisfy a certain quadratic growth condition (which is weaker than strong convexity). Moreover, we apply the developed  algorithms to the problem of multi-agent optimization on a graph, thus obtaining novel synchronous and asynchronous distributed methods. The proposed  algorithms are fully distributed in the sense that the updates and the stepsizes of each agent only depend on local information. In fact, no prior global coordination is required. 
Finally, we showcase an application of our algorithm in distributed formation control. 
\end{abstract}


\section{Introduction}
\label{sec:intro}

In this paper we consider the optimization problem 
  \begin{equation} \label{eq:equiv-prob}
  \underset{x\in\R^n}{\minimize}\ {f}({x})+{g}({x})+{{h}}({L}{x}),
  \end{equation}
 where $L$ is a linear mapping, $h$ and $g$ are proper, closed, convex functions (possibly nonsmooth), and $f$ is convex,  continuously differentiable with Lipschitz-continuous gradient. We further assume that the \emph{proximal mappings} associated with $h$ and $g$
are efficiently computable~\cite{combettes2011proximal}.
This setup is quite general and captures a wide range of applications in signal processing, machine learning and control. 

In problem \eqref{eq:equiv-prob}, it is 
typically assumed that the gradient of the smooth term $f$ is  $\beta_f$-Lipschitz for some nonnegative constant $\beta_f$.  
We consider  Lipschitz continuity of $\nabla f$ with respect to $\|\cdot\|_Q$ with $Q\succ0$ in place of the canonical norm (\textit{cf.} \eqref{eq:Lipz}). This is because in many applications of practical interest, a scalar Lipschitz constant fails to accurately capture the Lipschitz continuity of $\nabla f$.
A prominent example lies in distributed optimization, where $f$ is separable, \ie  $f(x)=\sum_{i=1}^{m}f_i(x_i)$. In this case, the metric $Q$ is taken block-diagonal  
with blocks containing the Lipschitz constants of the $\nabla f_i$'s. Notice that in such  settings considering a scalar Lipschitz constant  results in using the largest of the Lipschitz constants, 
 which leads to conservative stepsize selection and consequently slower convergence rates.

The main contributions of the paper are elaborated upon in four separate sections below.

\subsection{A New Primal-Dual Algorithm}\label{sec:introI}
In this work a new primal-dual algorithm, \Cref{Alg:synch-1}, is introduced for solving~\eqref{eq:equiv-prob}. 
The algorithm consists of  two proximal evaluations (corresponding to the two nonsmooth terms $g$ and $h$), one gradient evaluation (for the smooth term $f$), and one correction step (\textit{cf.} \cref{Alg:synch-1}). 
We adopt the general Lipschitz continuity assumption \eqref{eq:Lipz} in our convergence analysis, which is essential for avoiding conservative stepsize conditions that depend on the \emph{global} scalar Lipschitz constant.



 In \Cref{subsec:NewPD}, it is shown that the sequence generated by \Cref{Alg:synch-1} is  
$S$-Fej\'er monotone (with respect to the set of primal-dual solutions),\footnote{Given a symmetric positive definite matrix $S$, we say that a sequence is $S$-Fej\'er monotone  with respect to a set $C$ if it is Fej\'er monotone with respect to $C$ in the space equipped with $\langle{}\cdot{},{}\cdot{}\rangle_S$.} where $S$ is a  block diagonal positive definite matrix. This key property is exploited  in \Cref{subsec:CD} to develop a block-coordinate version of the algorithm with a general randomized activation scheme.

The connections of our method to other related primal-dual algorithms in the literature are discussed in  \Cref{subsec:VuCondat}. Most notably, we recap the V\~u-Condat scheme \cite{vu2013splitting,condat2013primal}, a popular algorithm used for solving the structured optimization problem~\eqref{eq:equiv-prob} (convergence of this method was established independently by V\~u~\cite{vu2013splitting } and Condat~\cite{condat2013primal}, by casting it in the form of the forward-backward splitting
). In the analysis of \cite{vu2013splitting,condat2013primal}, a scalar constant is used to capture the Lipschitz continuity of the gradient of $f$, thus resulting in potentially smaller stepsizes (and slower convergence in practice).  
In \cite{combettes2014forward}, the authors assume 
the more general Lipschitz continuity property \eqref{eq:Lipz}  by using a preconditioned variable metric forward-backward iteration. Nevertheless, the stepsize matrix is restricted to be proportional to $Q^{-1}$. 
In \Cref{subsec:VuCondat}, we show how the analysis technique for the new primal-dual algorithm can be used to recover the V\~u-Condat algorithm with general stepsize matrices, and highlight that this line of analysis leads to \emph{less restrictive} sufficient conditions on the selected stepsizes compared to \cite{vu2013splitting,condat2013primal,combettes2014forward}. More importantly, it is shown that unlike \Cref{Alg:synch-1}, the V\~u-Condat generated sequence is $S$-Fej\'er monotone, where $S$  is \emph{not diagonal}. As we discuss in the next subsection, this constitutes the main difficulty in devising a randomized version of the V\~u-Condat algorithm.

\subsection{Randomized Block-Coordinate Algorithm} \label{subsec:BC}

Block-coordinate (BC) minimization is a simple approach for tackling large-scale optimization problems. At each iteration,  a subset of the coordinates is updated while others are held fixed. \emph{Randomized} BC algorithms are  of particular interest, and can be divided into two main  categories: 

{\bf Type a}) comprises algorithms in which \emph{only one} coordinate is randomly activated and updated at each iteration. The BC versions of gradient \cite{nesterov2012efficiency} and proximal gradient methods \cite{richtarik2014iteration} belong in this category. A distinctive attribute of the aforementioned algorithms is the fact that the stepsizes are selected to be inversely proportional to the \emph{coordinate-wise} Lipschitz constant of the smooth term rather than the global one. This results in applying larger stepsizes in directions with smaller Lipschitz constant, and therefore leads to faster convergence. 

{\bf Type b}) contains methods where \emph{more than one} coordinate may be randomly activated and simultaneously updated \cite{bianchi2015coordinate,combettes2015stochastic}. Note that this class may also capture the single active coordinate (type a) as a special case. 
The convergence condition for this class of BC algorithms is typically the same as in the full algorithm.  
In \cite{bianchi2015coordinate,combettes2015stochastic} random BC is applied to $\alpha$-averaged operators by establishing stochastic Fej\'er monotonicity, while \cite{combettes2015stochastic} also considers quasi-nonexpansive operators. 
In \cite{pesquet2014class,bianchi2015coordinate} the authors obtain randomized BC algorithms based on the primal-dual scheme of V\~u and Condat; the main drawback is that, just as in the full version of these algorithms, the use of conservative stepsize conditions leads to slower convergence in practice.


The BC version of \Cref{Alg:synch-1} falls into the second class, \ie, it allows for a general randomized activation scheme (\textit{cf.} \cref{Alg:BC}). The proposed scheme converges under the same stepsize conditions as the full algorithm. As a consequence, in view of the characterization of Lipschitz continuity of $\nabla f$ in \eqref{eq:Lipz}, when $f$ is separable, \ie, $f(x)=\sum_{i=1}^m f_i(x_i)$, 
our approach leads to algorithms that depend on the \emph{local} Lipschitz constants (of $\nabla f_i$'s) rather than the global constant, thus assimilating the benefits of both categories. Notice that when $f$ is separable, the coordinate-wise Lipschitz continuity assumption of  \cite{nesterov2012efficiency,richtarik2014iteration,fercoq2015coordinate} is equivalent to \eqref{eq:Lipz} with $\beta_f=1$ and $Q=\blkdiag(\beta_1I_{n_1},\ldots,\beta_mI_{n_m})$, where $m$ denotes the number of coordinate blocks, $n_i$ denotes the dimension of the $i$-th coordinate block, and $\beta_i$ denotes the Lipschitz constant of $f_i$. In the general setting, \cite[Lem. 2]{nesterov2012efficiency} can be invoked to establish the connection between the metric $Q$ and the  coordinate-wise Lipschitz assumption. However, in many  cases (most notably the separable case) this lemma is conservative.

As mentioned in the prequel, in \Cref{subsec:VuCondat} the V\~u-Condat algorithm is recovered using the same analysis that leads to our proposed primal-dual algorithm. 
It is therefore natural to consider adapting the approach of  \Cref{subsec:CD} so as to devise a block-coordinate variant of the the V\~u-Condat algorithm. However, this is not possible given that the V\~u-Condat generated sequence is $S$-Fej\'er monotone, where $S$ is \emph{not diagonal} (\cf \eqref{S:vu-condat}), while the proof of \Cref{thm:Fejerlike-dg} relies heavily on the diagonal structure of $S$. This presents a distinctive merit of our proposed algorithm over the current state-of-the-art for solving problem~\eqref{eq:equiv-prob}.



In \cite{fercoq2015coordinate}, the authors propose a  randomized BC version of the V\~u-Condat scheme. Their analysis does not require the cost functions to be separable and utilizes a different Lyapunov function for establishing convergence. 
Notice  that the block-coordinate scheme of \cite{fercoq2015coordinate} updates a single coordinate at every iteration (\ie it is a type a) algorithm) as opposed to the more general random sweeping of the coordinates. 
 Additionally,  in the case of $f$ being separable, our proposed method (\cf \cref{Alg:BC}) assigns a block stepsize that is inversely proportional to $\tfrac{\beta_i}{2}$ (where $\beta_i$ denotes the Lipschitz constant for $f_i$), in place of ${\beta_i}$ required by \cite[Assum. 2.1(e)]{fercoq2015coordinate}: larger stepsizes are typically associated with faster convergence in primal-dual proximal algorithms.

\subsection{Linear Convergence} \label{subsec:linear}
A third contribution of the paper is establishing linear convergence for the full algorithm under an additional \emph{metric subregularity} condition for the monotone operator pertaining to the primal-dual optimality conditions (\textit{cf.} \cref{Thm: metricSub}).  
 	 For the BC version, the linear rate is established under a slightly stronger condition (\textit{cf.} \cref{thm:linearCon}). 
We further explicate the required condition in terms of the objective functions, with two special cases of prevalent interest: 
 a) when $f$, $g$ and $h$ satisfy a \emph{quadratic growth} condition (\textit{cf.} \cref{lem:QuadGrow}) (which is \emph{much weaker than strong convexity}) or b) when $f$, $g$ and $h$ are \emph{piecewise linear-quadratic} (\textit{cf.} \cref{lem:PLQ}), a common scenario in many applications such as LPs, QPs, SVM and fitting problems for a wide range of regularization functions; \textit{e.g.} $\ell_1$ norm, elastic nets, Huber loss and many more.  
 	




	Last but not least, it is shown that the monotone operator defining the primal-dual optimality conditions is metrically subregular if and only if the residual mapping (the operator that maps $z^k$ to $z^k-z^{k+1}$) is metrically subregular (\textit{cf.} \cref{lem:EquivMetr}). This connection enables the use of \Cref{lem:QuadGrow,lem:PLQ} to establish linear convergence for a large class of algorithms based on conditions for the cost functions.

\subsection{Distributed Optimization} 
As an important application, we consider a distributed structured optimization problem over a network of agents. 
In this context, each agent has its own private cost function of the form~\eqref{eq:equiv-prob}, while the communication among agents is captured by an undirected graph $\mathcal{G}=(\mathcal{V},\mathcal{E})$: 
 	\begin{align*}
 		\underset{{x}_1,\ldots,x_m}{\minimize}&\quad \sum_{i=1}^mf_i(x_i)+g_i(x_i)+h_i\left(L_ix_i\right)
 		\\\stt&\quad A_{ij}x_i + A_{ji}x_j=b_{(i,j)}
 		\qquad(i,j)\in \mathcal{E}. 
 	\end{align*}
We use $(i,j)$ to denote the unordered pair of agents $i$, $j$, and $ij$ to denote the ordered pair.  The goal is to solve the global optimization problem through local exchange of information.  
Notice that the linear constraints on the edges of the graph prescribe relations between neighboring agents' variables. This type of edge constraints was also considered in \cite{zhang2015bi}. It is worthwhile noting that for the special case of two agents $i=1,2$, with $f_i,h_i\equiv 0$, one recovers the setup for the celebrated \emph{alternating direction method of multipliers} (ADMM) algorithm. Another special case of particular interest is \emph{consensus optimization}, when $A_{ij} = I$, $A_{ji}=-I$ and $b_{(i,j)}=0$.  
A primal-dual algorithm for consensus optimization was introduced in \cite{latafat2016new} for the case of $f_i\equiv0$, where a transformation was used to replace the edge variables with node variables.

This multi-agent optimization problem arises in many contexts such as sensor networks, power systems, transportation networks, robotics, water networks, distributed data-sharing, etc. \cite{boyd2006randomized,jadbabaie2003coordination, raffard2004distributed}. In most of these applications, there are computation, communication and/or physical limitations on the system  that render centralized management infeasible. This motivates the \emph{fully} distributed synchronous and asynchronous algorithms developed in \Cref{sec:DistOpt}. 
Both versions are fully distributed in the sense that not only the iterations are performed locally, but also the stepsizes of each agent are selected based on local information without any prior global coordination (\cf \Cref{ass:Diststep}). 
The asynchronous variant of the algorithm is based on an instance of the randomized block-coordinate algorithm in \Cref{subsec:CD}. The protocol is as follows:  at each iteration, a) agents are activated at random, and independently from one another, b) active agents perform local updates, c)  
they communicate the required updated values to their neighbors and d) return to an idle state.

\subsection*{Notation and Preliminaries}

In this section, we introduce notation and definitions used throughout the paper; the interested reader is referred to    \cite{rockafellar2009variational,bauschke2011convex} for more details.   

For an extended-real-valued function $f$, we use $\dom f$ to denote its domain. For a set $C$, we denote its relative interior by $\ri C$. The identity matrix is denoted by $I_n\in\R^{n\times n}$.  
 For a symmetric positive definite matrix $P\in \R^{n\times n}$, we define the scalar product $\langle x,y \rangle_P=\langle x,Py \rangle$ and the induced norm 
$\|x\|_P = \sqrt{\langle x,x\rangle_P}$. 
For simplicity, we use matrix notation for linear mappings when no ambiguity occurs.

An operator (or set-valued mapping) $A:\R^n\rightrightarrows\R^d$ maps each point $x\in\R^n$ to a subset $Ax$ of $\R^d$. We denote the domain of $A$ by $\dom A=\{x\in\R^n\mid Ax\neq\emptyset\}$, its graph by $\gra A=\{(x,y)\in\R^n\times \R^d\mid y\in Ax\}$, 
the set of its zeros by $\zer A=\{x\in\R^n \mid 0\in Ax\}$, and the set of its fixed points by $\fix A = \{x\mid x\in Ax\}$. The mapping $A$ is called monotone if 
$
\langle x-x^\prime,y-y^\prime\rangle\geq0$ for all $(x,y),(x^\prime,y^\prime)\in\gra A$, 
and is said to be maximally monotone if its graph is not strictly contained by the graph of another monotone operator. The inverse of $A$ is defined through its graph: $\gra A^{-1}:=\{(y,x)\mid (x,y)\in\gra A\}$.
The \emph{resolvent} of $A$ is defined by $J_A:=(\id+A)^{-1}$, where $\id$ denotes the identity operator. 

Let $f:\R^n\to\Rinf\coloneqq\R\cup\{+\infty\}$ be a proper closed, convex function. 
Its subdifferential is the operator $\partial f: \R^n\rightrightarrows\R^n$  
$$
\partial f(x)=\{y\mid\forall z\in\R^{n},\,f(x)+\langle y,z-x\rangle\leq f(z)\}.
$$
It is well-known that the subdifferential of a convex function is maximally monotone. 
The resolvent of $\partial f$ is called the \emph{proximal operator} (or proximal mapping), and is single-valued. Let $V$ denote a symmetric positive definite matrix. The proximal mapping of $f$ relative to $\|\cdot\|_V$ is uniquely determined by the resolvent of $V^{-1} \partial f$: 
 \begin{align*}
 \prox_f^V(x)&\coloneqq (\id+ V^{-1}\partial f)^{-1}x\\&=\argmin_{z\in\R^n}\{ f(z) + \tfrac{1}{2}\|x-z\|_V^2\}.
 \end{align*}
The \emph{Fenchel conjugate} of $f$, denoted by $f^*$, is defined by 
$f^*(v)\coloneqq \sup_{x\in\R^n}\{ \langle v,x\rangle-f(x)\}$.
The \emph{Fenchel-Young} inequality states that $\langle x,u\rangle\leq f(x)+f^*(u)$ holds for all $x,u\in\R^n$; in the special case when $f=\tfrac{1}{2}\|\cdot\|^2_V$ for some symmetric positive definite matrix $V$, this gives:
\begin{equation}\label{eq:Fenchel-Young}
  	\langle x,u\rangle\leq \tfrac{1}{2}\|x\|_V^2+\tfrac{1}{2}\|u\|^2_{V^{-1}}.
  \end{equation}  

Let $X$ be a nonempty closed convex set. The indicator of $X$ is defined by $\delta_X(x)=0$ if $x\in X$, and $\delta_X(x)=\infty$ if $x
\notin X$. 
The distance from $X$ and the projection onto $X$ with respect to $\|\cdot\|_V$ are denoted by  $d_V(\cdot,X)$ and $\mathcal{P}_X^V(\cdot)$, respectively. 

We use $(\Omega,\mathcal{F},\mathbb{P})$ for defining a probability space, where $\Omega$, $\mathcal{F}$ and $\mathbb{P}$ denote the sample space, $\sigma$-algebra, and the probability measure. %
 Moreover, \emph{almost surely} is abbreviated as a.s.  

The sequence $\seq{w^k}$ is said to converge to $w^{\star}$ $Q$-linearly with $Q$-factor $\sigma\in(0,1)$\nolinebreak, if there exists $\bar{k}\in\N$ such that for all $k\geq\bar{k}$,   $\|w^{k+1}-w^{\star}\|\leq\sigma\|w^{k}-w^{\star}\|$. Furthermore, $\seq{w^k}$ is said to converge to $w^{\star}$ $R$-linearly if there exists a sequence of nonnegative scalars $({v_k})_{k\in\N}$ such that  $\|w^{k}-w^{\star}\|\leq v^k$ and $\seq{v_k}$ converges to zero $Q$-linearly.
		 		
			\section{A New Primal-Dual Algorithm}\label{subsec:NewPD}

In this section we present a primal-dual algorithm for problem \eqref{eq:equiv-prob}. 
  We adhere to the following assumptions throughout  \cref{subsec:NewPD,subsec:CD,sec:linConv}:
  \begin{ass}  \label{ass:1}
  	\quad 
	
  	\begin{enumerate}
  		\item $g:\R^{n}\to\Rinf$, $h:\R^{r}\to\Rinf$ are proper, closed, convex functions, and  $L:\R^{n}\to \R^r$ is a linear mapping. 
  		\item \label{ass:1-3}${f}:\R^{n}\to\R$ is  convex, continuously differentiable, and for some $\beta_f\in[0,\infty)$, $\nabla f$ is $\beta_f$-Lipschitz continuous with respect to the metric induced by $Q\succ0$ , \ie, 
  		\begin{equation}\label{eq:Lipz}
  		\|\nabla f(x)-\nabla f(y)\|_{Q^{-1}}\!\!\leq \!\beta_f\|x-y\|_{Q}\quad \forall x,y\in \R^n.
  		\end{equation}
		\item \label{ass:1-4} The set of solutions to~\eqref{eq:equiv-prob}  is nonempty. Moreover, there exists $x\in\ri \dom g$ such that $Lx\in\ri \dom h$. 
  	\end{enumerate}
  \end{ass} 
  In \Cref{ass:1-3}, the constant $\beta_f\ge 0$ is not absorbed into the metric $Q$ in order to also incorporate the case when $\nabla f$ is a constant (by setting $\beta_f=0$). 
  	   
  The dual problem is to 
   \begin{equation} \label{eq:dualprob}
  \minimize_{u\in\R^r} (g+f)^*(-L^\top u)+ h^*(u).
\end{equation}
With a slight abuse of terminology, we say that $({u}^\star,{x}^\star)$ is a \emph{primal-dual solution} (in place of dual-primal) if ${u}^\star$ solves the dual problem~\eqref{eq:dualprob} and ${x}^\star$ solves the primal problem~\eqref{eq:equiv-prob}.  We denote the set of  primal-dual solutions by $\mathcal{S}$. 
\Cref{ass:1-4} guarantees that the set of solutions to the dual problem is nonempty and the duality gap is zero \cite[Corollary 31.2.1]{rockafellar2015convex}. Furthermore, the pair $(u^\star,x^\star)$ is a primal-dual solution if and only if it satisfies:
 \begin{equation} \label{eq:primal-dual}
 \begin{cases}
 0\in\partial {{h}}^{*}({u})-{L}{x}, & \ \ \ \ \ \ \\0\in\partial {g}({x})+\nabla {f}(x)+{L}^{\top}{u}.& \ \ \ \ \ \ 
 \end{cases} 
 \end{equation}

We proceed to present the new primal-dual scheme \Cref{Alg:synch-1}. The motivation behind the name becomes apparent in the sequel after equation \eqref{eq:op1}. 
The algorithm involves two proximal evaluations (respective to the non-smooth terms $g,h$), and one gradient evaluation (for the Lipschitz-differentiable term $f$). The stepsizes in \Cref{Alg:synch-1} are chosen so as to satisfy the following assumption: 
  \begin{ass}[Stepsize selection]\label{ass:2}\quad 
 Both the dual stepsize matrix $\Sigma\in\R^{r\times r}$, and the primal stepsize matrix $\Gamma\in\R^{n\times n}$  
 are symmetric positive definite. In addition, they satisfy: 
      \begin{equation}\label{cond:synch}
    \Gamma^{-1}-{\tfrac{\beta_f}{2}}Q-L^\top \Sigma L\succ0.
    \end{equation} 
   \end{ass}
Selecting scalar primal and dual stepsizes, along with the standard definition of Lipschitz continuity, as is prevalent in the literature~\cite{vu2013splitting,condat2013primal}, can plainly be treated by setting $\Sigma=\sigma I_r$, $\Gamma = \gamma I_n$, and $Q=I_n$, whence from \eqref{cond:synch} we require that 
    \begin{equation*}
        \gamma < \frac{1}{\tfrac{\beta_f}{2}+ \sigma \|L\|^2}.
      \end{equation*} 
      \grayout{
      Specific rules for stepsize selection may be determined by taking into consideration the  application under study. In our limited experience we have found the following choices to perform well: 
      \begin{equation*} 
      \gamma = \frac{0.99}{\tfrac{\beta_f}{2}+ \alpha \|L\|},\quad  \sigma = \frac{\alpha}{\|L\|},
 \end{equation*}
with $\alpha = \tfrac{\beta_f}{\|L\|}$ when $\beta_f>\|L\|$, and $\alpha=1$ otherwise. The rationale being that when $\beta_f$ is larger than $\|L\|$ this choice allows larger $\sigma$ while having a smaller effect on $\gamma$; when $\beta_f$ is smaller that $\|L\|$ the two stepsizes are set to be almost equal.   }   
\begin{algorithm}[H]
      \caption{{\bf Tri}angularly Preconditioned {\bf P}rimal-{\bf D}ual algorithm (TriPD)}%
      \label{Alg:synch-1}%
      \begin{algorithmic}
      \algnotext{EndFor} 
        \item[]\textbf{Inputs:} $x^0\in\R^n$, $u^0\in\R^r$
        \For{$k=0,1,\ldots$}
        \State $\bar{u}^k  =\prox_{h^{*}}^{\Sigma^{-1}}(u^k+\Sigma L{x^k})$
        \State  ${x}^{k+1}  =\prox_{g}^{\Gamma^{-1}}(x^k-\Gamma\nabla f(x^k)-\Gamma L^{\top}\bar{u}^k)$
        \State $u^{k+1}=\bar{u}^k+\Sigma L({x}^{k+1}-x^k)$
        \EndFor
      \end{algorithmic}
    \end{algorithm}  
     \begin{rem}\label{rem:2products}
    Each iteration of \Cref{Alg:synch-1} requires one application of $L$ and one of $L^\top$ (even though it appears to require two applications of $L$). 
    The reason is that, at iteration $k$, only $L^\top \bar{u}^k$, $Lx^{k+1}$ need to be evaluated since $L(x^{k+1}-x^k)=Lx^{k+1}-Lx^k$ and $Lx^k$ was computed during the previous iteration. 
  \end{rem}
\Cref{Alg:synch-1} can be compactly written as:
  \begin{equation*}
      z^{k+1} = Tz^k,
  \end{equation*}
  where $z^k\coloneqq(u^k,x^k)$, and  the operator $T$ is given by:
  \begin{subequations} \label{eq:opT-nodelay}
    \begin{align}
    \bar{u} & =\prox_{h^{*}}^{\Sigma^{-1}}(u+\Sigma L{x})\\
    \bar{x} & =\prox_{g}^{\Gamma^{-1}}(x-\Gamma\nabla f(x)-\Gamma L^{\top}\bar{u})\\
    Tz & =(\bar{u}+\Sigma L(\bar{x}-x),\bar{x}).\label{eq:Tz}
    \end{align}
  \end{subequations}
  \begin{rem}[Relaxed iterations]\label{rem:Lambda}
    It is also possible to devise a \emph{relaxed} version of \Cref{Alg:synch-1} as follows:
  $$
  z^{k+1} = z^k+ \Lambda(Tz^k-z^k),
  $$
  where $\Lambda$ is a positive definite matrix and $\Lambda\prec 2I_{n+r}$. 
   For ease of exposition, we present the convergence analysis for the original version  (\ie for $\Lambda=I_{n+r}$). Note that the analysis carries through with minor modifications for relaxed iterations.  
  \end{rem} 
For compactness of exposition, we define the following operators:
  \begin{subequations} \label{monoper}
  	\begin{align}
  	A&:({u},{x})\mapsto(\partial {{h}}^*({u}),\partial {g}({x})),\label{eq:Aopt}\\
  	M&:({u},{x})\mapsto(-{L}x,{L}^\top{u}),\label{eq:Mopt}\\
  	C&:({u},{x})\mapsto(0,\nabla {f}({x})).\label{eq:Copt}
  	\end{align}
  \end{subequations}
  The optimality condition~\eqref{eq:primal-dual} can then be written in the equivalent form of the \emph{monotone inclusion}:
  \begin{equation} \label{eq:inclusion}
  0\in  Az+Mz+Cz \eqqcolon Fz ,
  \end{equation}
  where $z=({u},{x})$. Observe that the linear operator $M$ is monotone since it is skew-symmetric, \ie, $M^\top=-M$. It is also easy to verify that the operator $A$ is maximally monotone \cite[Thm. 21.2 and Prop. 20.23]{bauschke2011convex}, while operator $C$ is cocoercive,  being the gradient of $\tilde{f}(u,x)=f(x)$, and in light of \Cref{ass:1-3} and \cite[Thm. 18.16]{bauschke2011convex}.  

We further define 
    \begin{equation} \label{eq:PD-dg}
  P=\begin{pmatrix}
  \Sigma^{-1} & \frac{1}{2}{L}\\
  \frac{1}{2}{L}^\top & \Gamma^{-1}
  \end{pmatrix}, \quad 
  K=\begin{pmatrix}
  0 & -\tfrac{1}{2}{L}\\
  \tfrac{1}{2}{L}^\top & 0
  \end{pmatrix},
  \end{equation}
  and set $H=P+K$. It is plain to check that condition \eqref{cond:synch} implies that the symmetric matrix $P$ is positive definite (by a standard Schur complement argument). In addition, we set
  \begin{equation}\label{S:blkdiag}
  S=\blkdiag(\Sigma^{-1},\Gamma
  ^{-1}). 
  \end{equation} 
Using these definitions, the operator $T$ defined in \eqref{eq:opT-nodelay} can be written as:
 \begin{equation}\label{eq:op2}
Tz  \coloneqq z+S^{-1}(H+M^{\top})(\bar{z}-z),
 \end{equation}
 where 
 \begin{equation}
    \bar{z}  =(H+A)^{-1}(H-M-C)z.\label{eq:op1}
 \end{equation}
 This compact representation simplifies the convergence analysis.  
A key consideration for choosing $P$ and $K$ as in \eqref{eq:PD-dg} is to ensure that $H=P+K$ is lower block-triangular.  Notice that when $M\equiv0$, \eqref{eq:op2} can be viewed as a  \emph{triangularly preconditioned} forward-backward update, followed by a correction step.
This motivates the name \emph{TriPD}: {\bf Tri}angularly Preconditioned {\bf P}rimal-{\bf D}ual algorithm. Due to the triangular structure of $H$, the backward step $(H+A)^{-1}$ in \eqref{eq:op1} can be carried out sequentially: an updated dual vector $\bar{u}$ is computed (through proximal mapping) using $(u,x)$ and, subsequently, the primal vector $\bar{x}$ is  computed using $\bar{u}$ and $x$, \textit{cf.}~\eqref{eq:opT-nodelay}. 
Furthermore, it follows from~\eqref{eq:op2} that this choice makes $H+M^\top$ upper block-triangular which, alongside the diagonal structure of $S$, yields the efficiently computable update \eqref{eq:Tz} in view of:
  \begin{equation}\label{eq:HplusM}
  S^{-1}(H+M^\top)=\begin{pmatrix}
  I & \Sigma{L} \\
  0 & I 
  \end{pmatrix}.
  \end{equation}
   \begin{rem}\label{rem:AFBA}
The operator in~\eqref{eq:op2} is inspired from \cite[Alg. 1]{Latafat2017}, where operators of this form were introduced for devising a splitting method for solving general monotone inclusions of the form in \eqref{eq:inclusion}. We note, in passing, that the aforementioned algorithm entails 
an additional dynamic stepsize parameter ($\alpha_n$, therein). Although we may also adopt this here, for potentially improving the rate of convergence in practice, we opt not to: the reason is that in the context of  multi-agent optimization (that we especially target in this paper) such  design choice would require global coordination, that is contradictory to our objective of devising distributed algorithms. As a positive side-effect,  the convergence analysis is greatly simplified  compared to \cite[Sec. 5]{Latafat2017}. 
Besides, we use stepsize matrices (in place of scalar stepsizes) in \Cref{Alg:synch-1} along with the general Lipschitz continuity property (\cf  \Cref{ass:1-3}) as an essential means for avoiding conservative stepsizes, which is especially important for large-scale distributed optimization. 
\end{rem}
We proceed by showing that the set of primal-dual solutions coincides with the set of fixed points of $T$, $\fix T$: 
  \begin{equation} \label{eq:SfixT}
  \mathcal{S}=\{z\mid 0 \in Az + Mz +Cz\} = \fix T.
  \end{equation}  
To see this note that from~\eqref{eq:op2} and \eqref{eq:op1} we have: 
\begin{align*}
  z\in\fix T \iff & z=Tz \iff \bar{z}=z\nonumber \\
  \iff & (H+A)^{-1}(H-M-C)z = z \\ 
  \iff &  Hz -Mz-Cz \in Hz+Az  \iff z\in\mathcal{S}, 
\end{align*}
where in the second equivalence we used the fact that $S$ is positive definite and 
$\langle(H+M^\top)z,z \rangle \geq \|z\|^2_P$ for all $z\in\R^{n+r}$ (since $K$ is skew-adjoint and $M$ is monotone). 



  
    Next, let us define 
    \begin{equation} \label{eq:tildeP}
    \tilde{P}\coloneqq\begin{pmatrix}
    \Sigma^{-1} & -\frac{1}{2}{L} \\
    -\frac{1}{2}{L}^\top & \Gamma^{-1}-\tfrac{\beta_f}{4}Q 
    \end{pmatrix}.
    \end{equation}
Observe that (from Schur complement) \Cref{ass:2} is necessary and sufficient for $2\tilde{P}-S$ to be symmetric positive definite (\cf to the convergence result in \cref{thm:synch}). In particular, $\tilde{P}$ is positive definite since $S$ is positive definite.  

 The next lemma establishes the key property of the operator $T$ that is instrumental in our convergence analysis: 
 \begin{lem}\label{lem:nablaf-dg}
    Let \Cref{ass:1,ass:2} hold. Consider the operator $T$ in~\eqref{eq:opT-nodelay} (equivalently \eqref{eq:op2}).  
    Then for any $z^\star\in \mathcal{S}$ and any $z\in\R^{n+r}$ we have
    \begin{equation}\label{eq:main-innerprod}
    \|Tz-{z}\|_{\tilde{P}}^2 \leq \langle {z}-z^\star, {z}- Tz\rangle_S.
    \end{equation} 
  \end{lem} 
  \begin{proof}
		See \Cref{proof:lem:nablaf-dg}. 
   	\end{proof}

The next theorem establishes the main convergence result for \Cref{Alg:synch-1}. In specific, it is shown that the generated sequence is $S$-Fej\'er monotone. 
  We emphasize that the diagonal structure of $S$ is the key property used in developing the block-coordinate version of the algorithm in \Cref{subsec:CD}. 
  \begin{thm} \label{thm:synch}
  	Let \Cref{ass:1,ass:2} hold.  Consider the sequence $\seq{{z}^k}$ generated by~\Cref{Alg:synch-1}. 
  	The following Fej\'er-type inequality holds for all $z^\star\in\mathcal{S}$:
  \begin{equation}
    \|z^{k+1}-z^\star\|_{S}^2 {}\leq{}\|z^k-z^\star\|_{S}^2 -  \|z^{k+1}-z^k\|_{2{\tilde{P}}-S}^2\label{Fej:synch2}.
    \end{equation}
Consequently, $\seq{{z}^k}$ converges to some $z^\star\in\mathcal{S}$. 
  \end{thm}
    \begin{proof}
		See \Cref{proof:thm:synch}. 
   	\end{proof}

		
		\subsection{Related Primal-Dual Algorithms}\label{subsec:VuCondat} 
Recently, the design of primal-dual algorithms for solving  problem~\eqref{eq:equiv-prob} (possibly with  $f\equiv 0$ or $g\equiv0$) has received a lot of attention in the literature.  
Most of the existing approaches can be interpreted as  applications of one of the three main splittings techniques: forward-backward (FB), Douglas-Rachford (DR), and forward-backward-forward (FBF) splittings \cite{vu2013splitting,condat2013primal,briceno2011monotone+,combettes2012primal}, while others employ different tools to establish convergence  \cite{chambolle2011first,drori2015simple}.

A unifying analysis for primal-dual algorithms is proposed in \cite[Sec. 5]{Latafat2017}, where in place of FBS, DRS, or FBFS,  
a new three-term splitting, namely \emph{asymmetric forward-backward adjoint} (AFBA) is used to design primal-dual algorithms. In particular,  the algorithms of \cite{combettes2012primal,briceno2011monotone+,chambolle2011first,vu2013splitting,condat2013primal,drori2015simple} are recovered (under less restrictive stepsize conditions) and other new primal-dual algorithms are proposed. 
As discussed in \Cref{rem:AFBA} the  AFBA splitting \cite[Alg. 1]{Latafat2017} is the motivation behind the operator $T$ defined in~\eqref{eq:op2}. 
We refer the reader to \cite[Sec. 5]{Latafat2017} and \cite{Latafat2018chapter} for a detailed discussion on the relation between primal-dual algorithms.


Next we briefly discuss how the celebrated algorithm of V\~u and Condat \cite{condat2013primal,vu2013splitting} can be seen as fixed-point iterations of the operator $T$ in \eqref{eq:op2} for an appropriate selection of $S$, $P$, $K$.

  In \cite{condat2013primal} Condat considers  problem~\eqref{eq:equiv-prob}, while V\~u \cite{vu2013splitting} considers the following variant: 
 \begin{equation} \label{prob:gen}
   \underset{x\in\R^n}{\minimize}\ {f}({x})+{g}({x})+(h\oblong l)({L}{x}),
      \end{equation}
where $l$ is a strongly convex  function and $\oblong$ represents the infimal convolution~\cite{bauschke2011convex}.  
%
%
For this problem, 
an additional assumption is that the conjugate of $l$ is continuously differentiable, and $\nabla l^*$ is $\beta_l$-Lipschitz continuous with respect to a metric $G\succ0$, for some $\beta_l\ge 0$, \cf \eqref{eq:Lipz}. 
Note that it is possible to derive and analyze a variant of \Cref{Alg:synch-1} for \eqref{prob:gen}, however, we do not pursue this in this paper and focus on problem~\eqref{eq:equiv-prob} for clarity of exposition and length considerations. 

One can verify that the operator defining the fixed-point iterations in the V\~u-Condat algorithm is given by~\eqref{eq:op2} with $H=P+K$ and $S$ defined as follows: 
\begin{equation}\label{S:vu-condat}
S=\begin{pmatrix}
\Sigma^{-1} & {L}\\
{L}^\top & \Gamma^{-1}
\end{pmatrix},
	\end{equation}
	 \begin{equation*} 
	 P=\begin{pmatrix}
	 \Sigma^{-1} & {L}\\
	 {L}^\top & \Gamma^{-1}
	 \end{pmatrix}, \quad 
	 K=\begin{pmatrix}
	 0 & -{L}\\
	 {L}^\top & 0
	 \end{pmatrix}.
	 \end{equation*} 
For such selection of $S$, $P$, $K$, it holds that $S^{-1}(H+M^\top)=I$, whence in proximal form, the operator defined in~\eqref{eq:op2} becomes:
		\begin{align*}
		\bar{u} & =\prox_{h^{*}}^{\Sigma^{-1}}(u-\Sigma\nabla l^*(u)+\Sigma L{x})\\
		\bar{x} & =\prox_{g}^{\Gamma^{-1}}(x-\Gamma\nabla f(x)-\Gamma L^{\top}(2\bar{u}-u))\\
		Tz & =(\bar{u},\bar{x}).
		\end{align*}
	Observe the non-diagonal structure of $S$ for the V\~u-Condat algorithm in \eqref{S:vu-condat}, in contrast with the one for \Cref{Alg:synch-1} in \eqref{S:blkdiag}. 
For the sake of comparison with \cite{condat2013primal,vu2013splitting} we consider the relaxed iteration $z^{k+1}=z^k+\lambda (Tz^k-z^k)$ for some $\lambda\in(0,2)$, in this subsection (which we opted to exclude from \Cref{Alg:synch-1} solely for the purpose of simplicity). 

The analysis in \Cref{thm:synch} can be further used to establish convergence of the V\~u-Condat scheme for problem~\eqref{prob:gen} under the following sufficient conditions (in place of \Cref{ass:2}): 
	\begin{subequations}\label{eq:Vu-Condat}
	 		\begin{align}
			&\Sigma^{-1}-\tfrac{\beta_l}{2(2-\lambda)}G\succ0,\\
	 		&\Gamma^{-1}-\tfrac{\beta_f}{2(2-\lambda)}Q-L^{\top}\left(\Sigma^{-1}-\tfrac{\beta_l}{2(2-\lambda)}G\right)^{-1}L\succ0. 
	 		\end{align}
	\end{subequations}
Notice that when $l=\delta_{\{0\}}$ (\ie for problem~\eqref{eq:equiv-prob}), $l^*\equiv0$ whence $\beta_l=0$, and the condition simplifies to:
   \begin{equation*}
   \Gamma^{-1}-\tfrac{\beta_f}{2(2-\lambda)}Q-L^{\top}\Sigma L\succ0.
   \end{equation*} 
Given the stepsize condition \eqref{eq:Vu-Condat} the following Fej\'er-type inequality holds. 
  \begin{equation}
    \|z^{k+1}-z^\star\|_{S}^2 {}\leq{}\|z^k-z^\star\|_{S}^2 -  \lambda\|z^{k+1}-z^k\|_{2{\hat{P}}-\lambda S}^2,\label{Fej:synch3}
    \end{equation}
	  with $S$ defined in \eqref{S:vu-condat} and $\hat{P}$ given by:
		\begin{equation*} 
 		\hat{P}\coloneqq\begin{pmatrix}
 		\Sigma^{-1}-\tfrac{\beta_l}{4}G & {L} \\
 		{L}^\top & \Gamma^{-1}-\tfrac{\beta_f}{4}Q 
 		\end{pmatrix}.
 		\end{equation*}

This {generalizes} the result in \cite[Thm. 3.1]{condat2013primal}, \cite[Cor. 4.2]{vu2013splitting} and \cite[Prop. 5.1]{Latafat2017} where $Q=I$ and the stepsizes are assumed to be scalar.

Our main goal here was to demonstrate the non-diagonal structure of $S$ for the V\~u-Condat algorithm. 
In the sequel, we highlight that our analysis additionally leads to less conservative conditions as compared to \cite{vu2013splitting,condat2013primal,combettes2014forward}. 
Notice that the proofs in the aforementioned papers are based on casting the algorithm in the form of forward-backward iterations. Consequently, the stepsize condition obtained ensures that the underlying operator is \emph{averaged}. In contradistinction, the sufficient condition in \eqref{eq:Vu-Condat} only  ensures that the Fej\'er-type inequality \eqref{Fej:synch3} holds, which is sufficient for convergence. 
Therefore, even in the case of scalar stepsizes (as in \cite{vu2013splitting,condat2013primal}) condition \eqref{eq:Vu-Condat} allows for larger stepsizes compared to \cite{vu2013splitting,condat2013primal}. 

   %

In \cite{combettes2014forward,pesquet2014class} the authors propose a variable metric version of the algorithm with a preconditioning that accounts for the general Lipschitz metric. This is accomplished by fixing the stepsize matrix to be a constant times the inverse of the Lipschitz metric, and obtaining a condition on the constant.   
Our approach does not assume this restrictive form for the stepsize matrix; 
 even when such a restriction is imposed it allows for \emph{larger} stepsizes, thus achieving generally faster convergence. As an illustrative example, let us set $\Gamma=\mu Q^{-1}$ and $\Sigma= \nu G^{-1}$ for some $\mu,\nu>0$. For simplicity and without loss of generality, let $\beta_l=1$, $\beta_f=1$.  
Then \eqref{eq:Vu-Condat} simplifies to: 
   \begin{equation}
   (\mu^{-1}-\tfrac{1}{2(2-\lambda)})(\nu^{-1}-\tfrac{1}{2(2-\lambda)})Q-L^{\top}G^{-1}L\succ0,  \label{eq:ourVu}
   \end{equation}
   whereas the condition required in \cite{combettes2014forward,pesquet2014class} is $\lambda\in(0,1]$ and  
   \begin{equation}\label{eq:theirVu}
   \mathtight{\frac{\delta}{1+\delta}>\frac{\max\{\mu,\nu\}}{2}\; \textrm{with } \; \delta=\tfrac{1}{\sqrt{\nu\mu}}\|G^{-1/2}LQ^{-1/2}\|^{-1}-1.}
   \end{equation}
   It is not difficult to check that condition,~\eqref{eq:ourVu}, is always less restrictive than~\eqref{eq:theirVu}.  For instance, let $G^{-1/2}LQ^{-1/2}=I$  and set $\mu=1.5$, then \eqref{eq:ourVu} requires that $\nu<\tfrac{1}{6.5}$ whereas \eqref{eq:theirVu} necessitates that $\nu<\tfrac{1}{24}$.

	    \section{A Randomized Block-Coordinate Algorithm}\label{subsec:CD}

In this section, we describe a randomized block-coordinate variant of \Cref{Alg:synch-1} and discuss important special cases pertaining to the randomized coordinate activation mechanism. The convergence analysis is based on establishing stochastic Fej\'er monotonicity \cite{combettes2015stochastic} of the generated sequence. In addition, we establish linear convergence of the method under further assumptions in  \Cref{sec:linConv}. 

First, let us define a partitioning of the vector of primal-dual variables into $m$ blocks of coordinates. Notice that each block might include a subset of primal or  dual variables, or a combination of both. Respectively, let $U_i\in\R^{(n+r)\times(n+r)}$, for $i=1,\ldots,m$, be a diagonal matrix with $0$-$1$ diagonal entries that is used to select a subset of the coordinates (selected coordinates correspond to diagonal entries equal to $1$). We call such matrix an \emph{activation matrix}, 
 as it is used to activate/select a subset of coordinates to update.

Let $\Phi=\{0,1\}^m$ denote the set of binary strings of length $m$ (with the elements considered as column vectors of dimension $m$).   
At the $k$-th iteration, the algorithm draws a $\Phi$-valued random activation vector $\epsilon^{k+1}$ which determines which \emph{blocks of coordinates} will be updated. The $i$-th element of the vector $\epsilon^{k+1}$ is denoted as $\epsilon_i^{k+1}$:  the $i$-th block is updated at iteration $k$ if $\epsilon_i^{k+1}=1$. Notice that in general multiple blocks of coordinates may be concurrently updated.  
The conditional expectation ${\E}\left[\cdot\mid\mathcal{F}_k\right]$ is abbreviated by ${\E}_k\left[\cdot\right]$, where $\mathcal{F}_k$ is the filtration generated by $(\epsilon^1,\ldots,\epsilon^k)$. The following assumption summarizes the setup of the randomized coordinate selection. 
\begin{ass}\label{ass:4}\quad 
	
	\begin{enumerate}
		\item\label{ass:4-1}%
		$\{U_{i}\}_{i=1}^m$ are $0$-$1$ diagonal matrices and 
		\(
\sum_{i=1}^{m} U_{i} = I.
		\)
			\item\label{ass:4-2}%
			$\seq{\epsilon^k}$ is a sequence of i.i.d. $\Phi$-valued random vectors with 
			\begin{equation}\label{eq:probs}
				p_i\coloneqq\mathbb{P}(\epsilon^1_i=1)>0 \quad  i=1,\ldots,m.
			\end{equation}
		\item The stepsize matrices $\Sigma,\Gamma$ are diagonal. 
	\end{enumerate}
	\end{ass} 
The first condition implies that the activation matrices define a \emph{partition} of the coordinates, while the second that each partition is activated with a positive probability.

We further define the (diagonal) coordinate activation probability matrix $\Pi$ as follows:
\begin{equation}\label{eq:p}
\Pi\coloneqq\sum_{i=1}^{m}p_iU_{i}.
\end{equation}
For $\epsilon=(\epsilon_1,\ldots,\epsilon_m)$ we define the operator $\hat{T}^{(\epsilon)}$ by:
\begin{equation*}
	\hat{T}^{(\epsilon)}z \coloneqq  z + \sum_{i=1}^m \epsilon_i U_i(Tz-z),
\end{equation*}
where $T$ was defined in~\eqref{eq:opT-nodelay} (equivalently~\eqref{eq:op2}). Observe that this is a compact notation for the update of only the selected blocks.  
The randomized scheme is then written as an iterative  application of $\hat{T}^{(\epsilon^{k+1})}$ for $k=0,1,\ldots$ (this operator updates the active blocks of coordinates and leaves the others unchanged, \ie equal to their previous iterate values). The randomized block-coordinate scheme is summarized below. 
 		\begin{algorithm}[H]
 			\caption{Block-coordinate TriPD algorithm}
 			\label{Alg:BC}
 			\begin{algorithmic}
 			\algnotext{EndFor}
 				\item[]\textbf{Inputs:} $x^0\in\R^n$, $u^0\in\R^r$
 				\For{$k=0,1,\ldots$}
 				\State Select $\Phi$-valued r.v. $\epsilon^{k+1}$
 				\State $z^{k+1}=\hat{T}^{(\epsilon^{k+1})}z^k$ 
 				\EndFor
 			\end{algorithmic}
 		\end{algorithm}  

We emphasize that the randomized model that we adopt here is capable of capturing many stationary randomized activation mechanisms.  
To illustrate this, consider the following activation mechanisms (of specific interest in the realm of distributed multi-agent optimization, \cf \Cref{sec:DistOpt}):  
\begin{itemize}[leftmargin=*]
	\item \emph{Multiple coordinate activation}: 
	at each iteration, the $j$-th  coordinate block is randomly  activated with probability $p_j>0$ independent of other coordinates blocks. 
	This corresponds to the case that the sample space 
	is equal to $\Phi = \{0,1\}^m$. 
	The general distributed algorithm of \Cref{sec:DistOpt}  assumes this mechanism.   
	\item \emph{Single coordinate activation}: at each iteration, one  coordinate block is  selected, \ie the sample space 
	 is 
	{\mathtight\begin{equation}\label{eq:Phi2}
		\{(1,0,\ldots,0),(0,1,0,\ldots,0)\ldots,(0,\ldots,0,1)\}.
		\end{equation}}
	We assign probability $p_i$ to the event $\epsilon_i=1$ (and $\epsilon_j=0$ for $j\neq i$), whence the probabilities must satisfy $\sum_{i=1}^{m}p_i=1$. 
\end{itemize}
The next lemma establishes stochastic Fej\'er monotonicity for  the generated sequence, by directly exploiting the diagonal structure of $S$. The proof technique is adapted from \cite[Thm. 3]{bianchi2015coordinate} (see also \cite[Thm. 2]{iutzeler2013asynchronous}, \cite[Thm. 2.5]{combettes2015stochastic}), and is based on the  Robbins-Siegmund lemma \cite{robbins1985convergence}. 

 	\begin{thm} \label{thm:Fejerlike-dg} 
 		Let \Cref{ass:1,ass:2,ass:4} hold. Consider the sequence $\seq{{z}^k}$ generated by \Cref{Alg:BC}. The following Fej\'er-type inequality holds  for all $z^\star\in\mathcal{S}$:
 		 \begin{align}
 			\mathbb{E}_{k}\left[\|z^{k+1}-z^{\star}\|_{\Pi^{-1}S}^{2}\right]{}\leq{}&\|z^{k}-z^{\star}\|_{\Pi^{-1}S}^{2}\nonumber\\&-\|Tz^{k}-z^{k}\|_{2\tilde{P}-S}^{2}. \label{eq:stoFej}
 			\end{align}
 			Consequently, $\seq{{z}^k}$ converges a.s. to some $z^\star\in\mathcal{S}$. 
 	\end{thm}  
 	   \begin{proof}
		See \Cref{proof:thm:Fejerlike-dg}. 
   	\end{proof}
	It is important to emphasize that a naive implementation of  \Cref{Alg:BC} (with regards to the partitioning of primal-dual variables) may involve wasteful computations. 
	 As an example, consider a BC  algorithm in which, at every iteration, either all primal or all dual variables are updated. In such a case, if at iteration $k$ the dual vector is to be updated, both $x^{k+1}$, $u^{k+1}$ are computed (\cf \cref{Alg:synch-1}), whereas only $u^{k+1}$ is updated. 
	 This phenomenon is common to all primal-dual algorithms, and is due to the fact that the primal and dual updates need to be performed sequentially in the full version of the algorithm. As a consequence, the blocks of coordinates must be partitioned in such a way that computations are not  discarded, so that the iteration cost of a BC algorithm  is (substantially) smaller than computing the full operator $T$.  This choice relies entirely on the structure of the optimization problem under consideration. A canonical example of prominent practical interest is the setting of multi-agent optimization in a network (\cf \cref{sec:DistOpt}), where 
	 $L$ is not diagonal, $f$ and $g$ are separable, and additional coupling between (primal) coordinates is present through $h$, see \eqref{eq:dist}. 
	 In this example, the primal and dual coordinates are partitioned in such a way that no computation is discarded (\cf \cref{sec:DistOpt} for more details). 
%
%

We proceed with another example where the coordinates may be grouped such that the BC algorithm does not incur any wasteful computations: consider problem~\eqref{eq:equiv-prob} with $Lx = \blkdiag(L_{1}x_1,\ldots,L_{m}x_m)$, and $g$, $h$ separable functions \ie
	 \begin{equation*} 
   \underset{x\in\R^n}{\minimize}\ {f}({x}) + \sum_{i=1}^m \big( g_i(x_i) + h_i(L_{i}x_i)\big).
   \end{equation*}
   In this problem, the coupling between the (primal) coordinates 
    is carried via function $f$. For each $i=1,\ldots,m$, we can choose $U_i$ such that it selects the $i$-th primal-dual coordinate block $(u_i,x_i)$. Under such partitioning of coordinates, one may use \Cref{Alg:BC} with any random activation pattern satisfying \Cref{ass:4}. For example, for the case of multiple independently activated coordinates, 
   as discussed above, at iteration $k$ the following is performed
   \[
   \left\{\hspace*{-\algorithmicindent}~\!\!~
   \parbox{\linewidth}{%
    \begin{algorithmic}
        \State \scalebox{0.75}{\textbullet} each block $(u_i,x_i)$ is activated 
        with probability $p_i>0$ 
        \State \scalebox{0.75}{\textbullet} for active block(s) $i$ compute:
        \State $\bar{u}^k_{i}  =\prox_{\sigma h_i^{*}}(u^k_{i}+\sigma L_{i}x^k_i)$
        \State  ${x}^{k+1}_i  =\prox_{\gamma g_i}(x^k_i-\gamma\nabla_i f(x^k)-\gamma L^{\top}_{i}\bar{u}^k_i)$
        \State $u^{k+1}_i=\bar{u}^k_i+\sigma L_{i}({x}^{k+1}_i-x^k_i)$.      
      \end{algorithmic}}%
\right.\]
More generally, when $g$ and $h$ are separable in problem \eqref{eq:equiv-prob}, and $L$ is such 
that either each (block) row only has one nonzero element or each (block) column has one nonzero element, then the coordinates can be grouped together in such a way that no wasteful computations occur: in the first case the primal vector $x_i$ and all dual vectors $u_j$ that are required for its computation are selected by $U_i$ (with the role of primal and dual reversed in the second case). 

 		\begin{rem}
 				Note that in \Cref{Alg:BC} the probabilities $p_i$ are taken fixed, \ie,  the matrix $\Pi$ is constant throughout the iterations. This is a non-restrictive assumption and can be relaxed by considering iteration-varying probabilities $p_i^k$ in \eqref{eq:probs} and modifying \Cref{Alg:BC} by setting: 
 				$$	z^{k+1} =  z^k + \sum_{i=1}^m  \tfrac{\epsilon_i^{k+1}}{mp_i^{k+1}} U_i(Tz^k-z^k).$$ 
 				Let $\Pi^k$ denote the probability matrix defined as in~\eqref{eq:p} using $p_i^{k}$. Then, by arguing as in \Cref{thm:Fejerlike-dg}, it can be shown that 
 				the following stochastic Fej\'er monotonicity holds for the modified sequence:
 					\begin{align*}
 					\mathbb{E}_{k}\left[\|z^{k+1}-z^{\star}\|_{S}^{2}\right]\leq&\|z^{k}-z^{\star}\|_{S}^{2}\nonumber\\&-\|Tz^{k}-z^{k}\|_{\tfrac{2}{m}\tilde{P}-\tfrac{1}{m^2}S(\Pi^{k+1})^{-1}}^{2}.
 					\end{align*}     
 				\end{rem}

 				%

	    \section{Linear Convergence}\label{sec:linConv}
 

In this section, we establish linear convergence of \Cref{Alg:synch-1,Alg:BC} under additional conditions on the cost functions $f$, $g$ and $h$. 
To this end, we show that linear convergence is attained if  the monotone operator $F=A+M+C$  defining the primal-dual optimality conditions (\textit{cf.} \eqref{eq:inclusion}) is \emph{metrically subregular} (globally metrically subregular in the case of \Cref{Alg:BC}). A notable consequence of our analysis is the fact that linear convergence is attained when the cost functions either a) belong in the class of \emph{piecewise linear-quadratic} (PLQ) convex functions or b) when they satisfy a certain \emph{quadratic growth}  condition (which is much weaker than strong convexity). 
Moreover, notice that in the case of PLQ the solution need not be unique (\cf  \cref{Thm: metricSub,thm:linearCon}). 

We first recall the notion of \emph{metric subregularity} \cite{dontchev2009implicit}.
  \begin{defin}[Metric subregularity] \label{metricsub}
  	A set-valued mapping $F:\R^n\rightrightarrows\R^d$ is \emph{metrically subregular} at $\bar{x}$ for $\bar{y}$ if $(\bar{x},\bar{y})\in\gra F$ and there exists a positive  constant $\eta$ together with 
  	 a \emph{neighborhood of subregularity} $\mathcal{U}$ of $\bar{x}$ such that 
   	\begin{align*}
  	d(x,F^{-1}\bar{y})\leq{}& \eta d(\bar{y},Fx){ \;\; \forall} 
  	x\in \mathcal{U}. 
  	\shortintertext{
  	If the following stronger condition holds }
    \|x-\bar{x}\|\leq{}& \eta d(\bar{y},Fx) \;\; \forall  
    x\in \mathcal{U}, 
    \end{align*}
       then $F$ is said  to be \emph{strongly subregular} at $\bar{x}$ for $\bar{y}$.


   Moreover, we say that $F$ is globally (strongly) subregular at $\bar{x}$ for $\bar{y}$ if (strong) subregularity holds with $\mathcal{U}=\R^n$. 
  \end{defin}
   We refer the reader to \cite[Chap. 9]{rockafellar2009variational}, \cite[Chap. 3]{dontchev2009implicit} and \cite[Chap. 2]{ioffe2017variational} for further discussion on metric subregularity.   



   Metric subregularity of the subdifferential operator has been studied thoroughly and is equivalent to the \emph{quadratic growth condition} \cite{aragon2008characterization,drusvyatskiy2016error} defined next.  
    In particular, for a proper closed convex function $f$, the subdifferential $\partial f$ is metrically subregular at $\bar{x}$ for $\bar{y}$ with $(\bar{x},\bar{y})\in\gra \partial f$
   if and only if there exists a positive constant $c$ and a neighborhood $\mathcal{U}$ of $\bar{x}$ such that the following 
   growth condition holds \cite[Thm. 3.3]{aragon2008characterization}:
   \begin{equation*}
   f(x)\geq f(\bar{x}) + \langle\bar{y},x-\bar{x}\rangle+cd^2(x,(\partial{f})^{-1}(\bar{y}))\quad \forall x\in \mathcal{U}
   \end{equation*} 
   Furthermore, $\partial f$ is strongly subregular at $\bar{x}$ for $\bar{y}$ with $(\bar{x},\bar{y})\in\gra \partial f$, if and only if there exists a positive constant $c$ and a neighborhood $\mathcal{U}$ of $\bar{x}$ such that \cite[Thm. 3.5]{aragon2008characterization}:
   \begin{equation}\label{eq:quadgrowth-s}
   f(x)\geq f(\bar{x}) + \langle\bar{y},x-\bar{x}\rangle+c\|x-\bar{x}\|^2\quad \forall x\in \mathcal{U}
   \end{equation}    
    Note that strongly convex functions satisfy~\eqref{eq:quadgrowth-s}, but~\eqref{eq:quadgrowth-s} is \emph{much weaker} than strong convexity, as it is a \emph{local condition}: it only holds in a neighborhood of $\bar{x}$, and also only for $\bar{y}$. 
   
The lemma below provides a sufficient condition for metric subregularity of the monotone operator $A+M+C$, in terms of  
strong subregularity of $\nabla f+\partial g$ and $\partial h^*$ (equivalently the quadratic growth of $f+g$ and $h^*$, \textit{cf.} \eqref{eq:quadgrowth-s}) as stated in the following assumption:
\begin{ass}[Strong subregularity of $\nabla f+\partial g$ and $\partial h^*$] \label{ass:Quad} \label{ass:quadgrowthcost}
    There exists $z^\star=(u^\star,x^\star)\in\mathcal{S}$ satisfying:
    \begin{enumerate}
       \item $\nabla f+\partial g$ is strongly subregular at $x^\star$  for $-L^\top u^\star$,
       \item $\partial h^*$ is strongly subregular at $u^\star$  for $Lx^\star$. 
     \end{enumerate} 
    We say that $f$, $g$ and $h$ satisfy this assumption \emph{globally} if
     the strong subregularity assumption of $\nabla f+\partial g$ and $\partial h^*$  both hold globally (\textit{cf.} \Cref{metricsub}). 
\end{ass} 
 In particular, \Cref{ass:Quad} holds globally if either $f$ or $g$ (or both) are strongly convex and $h$ is continuously differentiable with Lipschitz continuous gradient, \ie, $h^*$ is strongly convex. 
%
 \begin{lem} \label{lem:QuadGrow}
   Let \Cref{ass:1,ass:Quad} hold. Then $F=A+M+C$ (\textit{cf.} \eqref{monoper}) is strongly subregular at ${z}^\star$ for $0$. 
    Moreover, if $f$, $g$ and $h$ satisfy \Cref{ass:Quad} \emph{globally}, then $F$ is globally strongly subregular at ${z}^\star$ for $0$. In both cases the set of primal-dual solutions is a singleton, $\mathcal{S}=\{z^\star\}$.
    \end{lem}
     \begin{proof}
    See \Cref{proof:lem:QuadGrow}. 
    \end{proof}
 Our next objective is to show that $A+M+C$ is globally metrically subregular when the functions $f$, $g$ and $h$ are \emph{piecewise linear-quadratic} (PLQ). Note that this assumption does not imply that the set of solutions $\mathcal{S}$ is a singleton, nevertheless, linear convergence can still be established. Let us recall the definition of PLQ functions \cite{rockafellar2009variational}:  
 \begin{defin}[Piecewise linear-quadratic]
 	A function $f:\R^n\to\Rinf$ is called piecewise linear-quadratic (PLQ) if its domain can be represented as the union of finitely many polyhedral sets, and in each such set $f(x)$ is given by an expression of the form $\frac{1}{2}\langle x, Qx\rangle +\langle d, x\rangle+c$, 
 	for some $c\in\R$, $d\in\R^n$, and symmetric matrix $Q\in\R^{n\times n}$. 
 \end{defin}
 The class of PLQ functions is closed under scalar multiplication, addition, conjugation and Moreau envelope 
  \cite{rockafellar2009variational}.  
  A wide range of functions used in optimization applications belong to this class, for example: affine functions, quadratic forms, indicators of polyhedral sets, polyhedral norms (\eg, the $\ell_1$-norm), and regularizing functions  such as elastic net, Huber loss, hinge loss, to name a few.  	
  
\begin{lem}\label{lem:PLQ}
	Let \Cref{ass:1} hold. In addition, assume that $f$, $g$ and $h$ are piecewise linear-quadratic. Then $F= A+M+C$ (cf. \eqref{monoper}) is metrically subregular with the same constant $\eta$ at any $z$ for any $v$ with $(z,v)\in\gra F$. 
\end{lem}
%
%
  \begin{proof}
    See \Cref{proof:lem:PLQ}. 
    \end{proof}

Our main convergence rate results are provided in \Cref{Thm: metricSub,thm:linearCon}. 
In this context, \Cref{lem:PLQ,lem:QuadGrow} are used to establish sufficient conditions in terms of the cost functions. 
We omit the proof  of \Cref{Thm: metricSub} for length considerations. The proof is similar to that of \Cref{thm:linearCon}, the main difference being that in \Cref{Thm: metricSub} \emph{local} (as opposed to global) metric subregularity is used: due to the Fej\'er-type inequality \eqref{Fej:synch2}, $\bar{z}^k$ will eventually be contained in a neighborhood of metric subregularity, where inequality \eqref{eq:metricSub} applies.

  \begin{thm}[Linear convergence of \cref{Alg:synch-1}] \label{Thm: metricSub}
    Consider \Cref{Alg:synch-1} under the assumptions of~\Cref{thm:synch}.  Suppose that $F=A+M+C$ is {metrically subregular} at all $z^\star\in\mathcal{S}$ for $0$. 
    Then $\seq{d_{S}({z}^k,\mathcal{S})}$ converges $Q$-linearly to zero, and  $(z^k)_{k\in\Nn}$ converges $R$-linearly to some $z^\star\in\mathcal{S}$. 

  In particular, the metric subregularity assumption holds and the 
	result follows if either one of the following holds: 
    \begin{enumerate}
      \item either $f$, $g$ and $h$ are PLQ, 
      \item \label{thm:linearCon1} or $f$, $g$ and $h$ satisfy \Cref{ass:quadgrowthcost}, in which case the solution is unique. 
    \end{enumerate}
  \end{thm}
   \begin{thm}[Linear convergence of \cref{Alg:BC}]  \label{thm:linearCon}
      Consider \Cref{Alg:BC} under the assumptions of \Cref{thm:Fejerlike-dg}.
      Suppose that $F=A+M+C$ is globally {metrically subregular} 
       for $0$ (\textit{cf.} \cref{metricsub}), \ie, there exists $\eta>0$ such that 
      \begin{equation*}
         d(z,F^{-1}0)\leq \eta d(0,Fz) \quad \forall z\in\R^{n+r}.
      \end{equation*}
      Then $\seq{\E\left[d^2_{\Pi^{-1}S}({z}^{k},\mathcal{S})\right]}$ converges $Q$-linearly  to zero. 

The same holds if     
       \begin{enumerate}
            \item either $f,g,h$ are PLQ and there exists a compact set $\mathcal{C}$ such that  
            $\seq{z^k}\subseteq \mathcal{C}$ (as is the case if $\dom g$ and $\dom h^*$ are compact),%
         \item \label{thm:linearCon2} or $f$, $g$ and $h$ satisfy \Cref{ass:quadgrowthcost} globally, in which case the solution is unique.  
       \end{enumerate}
    \end{thm} 
     \begin{proof}
    See \Cref{proof:thm:linearCon}. 
    \end{proof}


In the recent work \cite{Liang2016} the authors establish linear convergence in the framework of non-expansive operators under the assumption that the residual mapping defined as $R=\id-T$ is metrically subregular.
However, such a condition is not easily verifiable in terms of conditions on the cost functions.  
In the next lemma, we show that $R$ is metrically subregular if and only if the monotone operator $F$ is metrically subregular. This result connects the two assumptions and is interesting in its own right. More importantly, it enables the use of \Cref{lem:QuadGrow,lem:PLQ} for establishing linear convergence for a wide array of problems. 
  \begin{lem}\label{lem:EquivMetr}
    Let \Cref{ass:1,ass:2} hold. Consider the operator $T$ defined in \eqref{eq:op2} and a point $z^\star\in\mathcal{S}$. 
   Then $F=A+M+C$ (\textit{cf.} \eqref{monoper}) is metrically subregular at $z^\star$ for $0$ if and only if the residual mapping $R\coloneqq \id-T$  is metrically subregular at $z^\star$ for $0$. 
  \end{lem}
     \begin{proof}
    See \Cref{proof:lem:EquivMetr}. 
    \end{proof}

   \section{Distributed 
   	 Optimization}\label{sec:DistOpt}

In this section, we consider a general formulation for multi-agent optimization over a network, and leverage \Cref{Alg:synch-1,Alg:BC} to devise both synchronous and randomized asynchronous distributed primal-dual algorithms. The setting is as follows. We consider an undirected graph  $\mathcal{G}=(\mathcal{V},\mathcal{E})$ over a vertex set $\mathcal{V}=\{1,\ldots,m\}$ with edge set $\mathcal{E}\subset{\mathcal{V}\times \mathcal{V}}$. 
Each vertex 
is associated with a corresponding \emph{agent}, which is assumed to have a local memory and computational unit, and can only communicate with its neighbors. 
We define the \emph{neighborhood} of agent $i$ by $\mathcal{N}_i\coloneqq \{j|(i,j)\in \mathcal{E}\}$. 
We use the terms vertex, agent, and node interchangeably.
The goal is to solve the following global optimization problem in a distributed fashion: 
 \begin{subequations} \label{prob:GenProblem}
  	\begin{align}
  		\underset{x_1,\ldots,x_m}{\minimize}&\quad \sum_{i=1}^mf_i(x_i)+g_i(x_i)+h_i\left(L_ix_i\right)\label{eq:gp-1}\\\stt&\quad A_{ij}x_i + A_{ji}x_j=b_{(i,j)}
  		\qquad(i,j)\in \mathcal{E}, \label{eq:gp-2}
  	\end{align}
  \end{subequations}
 where $x_i\in\R^{n_i}$. The cost functions $f_i$, $g_i$, $h_i\circ L_i$ are taken \emph{private} to agent/node $i\in\mathcal{V}$, \ie our distributed methods operate solely by exchanging local  variables among neighboring nodes that are unaware of each other's objectives. The coupling in the problem is represented through the edge constraints \eqref{eq:gp-2}. 


 Throughout this section the following assumptions hold:
\begin{ass}  \label{ass:5}
	For each $i=1,\ldots,m$: 
	\begin{enumerate}
		 \item For $j\in \mathcal{N}_i$, $b_{(i,j)}\in\R^{l_{(i,j)}}$ and $A_{ij}\in\R^{n_i} \to \R^{l_{(i,j)}}$ is a linear mapping. 
		\item $g_i:\R^{n_i}\to\Rinf$, $h_i:\R^{r_i}\to\Rinf$ are proper closed convex functions, and  $L_i:\R^{n_i}\to \R^{r_i}$ is a linear mapping.
		\item \label{ass:5-3}$f_i:\R^{n_i}\to\R$ is  convex, continuously differentiable, and for some $\beta_i\in[0,\infty)$, $\nabla f_i$ is $\beta_i$-Lipschitz continuous  with respect to the metric $Q_i\succ0$, \ie, 
		$$\|\nabla{f}_i(x)-\nabla{f}_i(y)\|_{Q_i^{-1}}\leq\beta_i\|x-y\|_{Q_i}\quad x,y\in\R^{n_i}.$$
		\item \label{ass:conctd} The graph $\mathcal{G}$ is connected.
	\item\label{ass:2-4} The set of solutions of~\eqref{prob:GenProblem} is nonempty. Moreover,  there exists $x_i\in\ri\dom g_i$ such that $L_ix_i\in\ri \dom h_i$, for $i=1,\ldots,m$,  and $A_{ij}x_i+A_{ji}x_j=b_{(i,j)}$ for  $(i,j)\in\mathcal{E}$.
		\end{enumerate} 
\end{ass} 

\newcommand{\INDSTATE}[1][1]{\State\hspace{\algorithmicindent}}  
\begin{algorithm*}
	\caption{Synchronous \& asynchronous versions of TriPD-Dist algorithm}
	\label{Alg:dist}
	\begin{algorithmic}
	 			\algnotext{EndFor}

				\newcommand{\alignThm}[2][c]{%
					\setlength\myLength{\widthof{${\bar{w}}_{(i,j)}^k$}}
					\ifstrequal{#1}{c}{
						\hspace*{0.5\myLength}
						\mathclap{#2}
						\hspace*{0.5\myLength}
					}{}
					\ifstrequal{#1}{l}{
						\mathrlap{#2}
						\hspace*{\myLength}
					}{}
					\ifstrequal{#1}{r}{
						\hspace*{\myLength}
						\mathllap{#2}
					}{}
				}%
		\item[]\textbf{Inputs:} 
		 $x_i^0\in\R^{n_i}$, $y_i^0\in\R^{r_i}$, for $i=1,\ldots,m$, and $w_{(i,j),i}\in\R^{l_{(i,j)}}$ for $j\in\mathcal{N}_i$.
		\For{$k=0,1,\ldots$} 
		\item[]  \fbox{
			\begin{minipage}[t][1.005\height]{0.30\textwidth-2\fboxsep-2\fboxrule\relax}
				\textbf{{I}: Synchronous version} 
				\vspace{10pt}
			
				\textbf{for} all agents $i=1,\ldots,m$ \textbf{do}
			\end{minipage} } 
			\fbox {
			\begin{minipage}[t][1.005\height]{0.61\textwidth-2\fboxsep-2\fboxrule\relax}
				\textbf{II: Asynchronous version} 

				\vspace{-1pt} Each agent $i=1.\ldots,m$ is activated independently with probability $p_i>0$
				
				\vspace{-1pt}{\textbf{for} all active agents \textbf{do}} 
			\end{minipage} 
		}
		\vspace*{3pt}
				\State {\bf Local updates}:
		\INDSTATE $\alignThm{{\bar{w}}_{(i,j),i}^k}=\tfrac{1}{2}\big({w}_{(i,j),i}^k+{w}_{(i,j),j}^{k}\big)+\tfrac{\kappa_{(i,j)}}{2}\big(A_{ij}x_{i}^k+A_{ji}{x}_{j}^{k}-b_{(i,j)}\big),\quad\forall j\in \mathcal{N}_i$
		\INDSTATE $\alignThm{\bar{{y}}_i^k}=\prox_{\sigma_i{{h}_i}^\star}\big({y}_i^k+\sigma_i{L}_i{x}_i^k\big)$
		\INDSTATE $ \alignThm{{{x}}_i^{k+1}}=\prox_{\tau_i{g}_i}\big({x}_i^k-\tau_i{L}_i^\top\bar{{y}}_i^k-\tau_i\sum_{j\in\mathcal{N}_i}A_{ij}^\top\bar{w}_{(i,j),i}^k-\tau_i\nabla {f_i}({x_i^k})\big)$
		\INDSTATE $\alignThm{y_i^{k+1}}=\bar{{y}}_i^k+\sigma_i{L}_i({x}_i^{k+1}-x_i^k)$
		\INDSTATE $\alignThm{w_{(i,j),i}^{k+1}}= \bar{w}_{(i,j),i}^k+\kappa_{(i,j)}{A_{ij}}({x}_i^{k+1}-x_i^k), \quad \forall j\in \mathcal{N}_i$
		\vspace*{2pt}
		\State {\bf Transmission of information}:
		\INDSTATE Send $A_{ij}x_i^{k+1}$, $w_{(i,j),i}^{k+1}$ to agent $j$,  $\forall j\in\mathcal{N}_i$
		\EndFor
	\end{algorithmic}
\end{algorithm*}  
Each agent $i\in \mathcal{V}$ maintains its own local primal variable $x_i\in\R^{n_i}$ and dual variables $y_i\in \R^{r_i}$, and $w_{(i,j),i}\in\R^{l_{(i,j)}}$ (for each $j\in \mathcal{N}_i$), where the former  is related to the linear mapping $L_i$, and the latter is the  local dual variable of agent $i$ corresponding to the edge-constraint~\eqref{eq:gp-2}. 
It is important to note that the updates in \Cref{Alg:dist} are performed locally through communication with neighbors:
 the only information that agent $i$ shares with its neighbor $j\in \mathcal{N}_i$ is the quantity $A_{ij}x_i$, along with edge variable $w_{(i,j),i}$, while all other variables are kept \emph{private}.

The proposed distributed protocol features both a synchronous as well as an  asynchronous implementation. In the synchronous version, at every iteration, all the agents update their variables. In the randomized asynchronous implementation, only a subset of randomly activated agents perform updates, at each iteration, and they do so using their local variables as well as information previously communicated to them by their neighbors. After an update is performed, in both cases, updated values are communicated to neighboring agents. 
Notice that  the asynchronous scheme corresponds to the case of multiple coordinate blocks activation 
in \Cref{Alg:BC}. Other activation schemes can also be considered, and our convergence analysis plainly carries over;  notably, the single agent activation which corresponds to the asynchronous model of 
\cite{tsitsiklis1986distributed,RK,RK2} in which agents are assumed to `wake-up' based on independent exponentially distributed tick-down timers.

  
%

Furthermore, in \Cref{Alg:dist} each agent $i$ keeps positive local stepsizes $\sigma_i$, $\tau_i$ and $\left(\kappa_{(i,j)}\right)_{j\in\mathcal{N}_i}$. The edge weights/stepsizes $\kappa_{(i,j)}$ may alternatively be interpreted as inherent parameters of the communication graph. For example, they may be used to capture edge's  `fidelity,' \eg the channel quality in a communication link. The stepsizes are assumed to satisfy the following \emph{local} assumption that is sufficient for the convergence of the algorithm (\cf \cref{thm:dist,thm:linearCon-2}).    
\begin{ass}[Stepsizes of \Cref{Alg:dist}] \label{ass:Diststep}
\quad 
	\begin{enumerate}
		\item (node stepsizes) Each agent $i$ keeps two positive stepsizes $\sigma_i$, $\tau_i$. 
		\item (edge stepsizes) A positive stepsize $\kappa_{(i,j)}$ is associated with edge $(i,j)\in\mathcal{E}$, and is shared between agents $i$, $j$. 
		\item \label{cond:cord} (convergence condition) The stepsizes satisfy the following \emph{local} condition  
			\begin{equation*}
				\tau_i < \frac{1}{\tfrac{\beta_i\|Q_i\|}{2}+ \|\sigma_i L_i^\top L_i+\sum_{j\in\mathcal{N}_i}\kappa_{(i,j)}A_{ij}^\top A_{ij}\|}.
				\end{equation*}
	\end{enumerate} 
\end{ass} 
	According to \Cref{cond:cord} the stepsizes $\tau_i,\sigma_i$ for each agent only depend on the local parameters $\beta_i$, $\|Q_i\|$, the edge weights, $\kappa_{(i,j)}$ and the linear mappings $L_i$, and  $A_{ij}$, which are all known to agent $i$; therefore the stepsizes can be selected locally, in a decentralized fashion. 

We proceed by casting the multi-agent optimization problem~\eqref{prob:GenProblem} in the form of the structured optimization problem \eqref{eq:equiv-prob}. In doing so, we describe how \Cref{Alg:dist} is derived as an instance of \Cref{Alg:synch-1,Alg:BC}. 

 Define the linear operator
  $$N_{(i,j)}:\mathsf{x}\mapsto (A_{ij}x_{{i}},A_{ji}x_{{j}}),$$
  and $\mathsf{N}\in\R^{2\sum_{(i,j)\in\mathcal{E}}l_{(i,j)}\times\sum_{i=1}^m n_i}$ by stacking $N_{(i,j)}$: 
 $$\mathsf{N}:\mathsf{x}\mapsto(N_{(i,j)}\mathsf{x})_{{(i,j)}\in \mathcal{E}}.$$
 Its transpose 
 is given by:
 $$\mathsf{N}^\top:(w_{(i,j)})_{(i,j)\in \mathcal{E}}\mapsto \tilde{\mathsf{x}}=\sum_{(i,j)\in \mathcal{E}}N^\top_{(i,j)}w_{(i,j)},$$
 with $\tilde{x}_{i}=\sum_{j\in\mathcal{N}_i}A^\top_{ij}w_{(i,j),i}$.  We have set $w_{(i,j)} = (w_{(i,j),i},w_{(i,j),j})$, \ie, we consider two dual variables (of dimension $l_{(i,j)}$) for each edge constraint, where $w_{(i,j),i}$ is maintained by agent $i$ and $w_{(i,j),j}$ by agent $j$.

 Consider the set 
 \begin{equation*} 
 C_{(i,j)}=\{(z_1,z_2)\in\R^{l_{(i,j)}}\times\R^{l_{(i,j)}}\mid z_1+z_2=b_{(i,j)}\}.
 \end{equation*}
  Then problem~\eqref{prob:GenProblem} can then be re-written as: 
 \begin{align}
 \minimize \sum_{i=1}^m&f_i(x_i)+g_i(x_i)+h_i\left(L_ix_i\right) \nonumber\\+& \sum_{(i,j)\in \mathcal{E}} \delta_{C_{(i,j)}}(N_{(i,j)}\mathsf{x})\label{eq:probequi}
 \end{align}

Let $C=\bigtimes_{(i,j)\in \mathcal{E}}C_{(i,j)}$, $L=\blkdiag(L_1,\ldots,L_m)$, and  $\mathsf{L}\mathsf{x}=(L\mathsf{x},\mathsf{N}\mathsf{x})=:(\mathsf{\tilde{y}},\mathsf{\tilde{w}})\in\R^{n_d}$ with $n_d=2\sum_{(i,j)\in \mathcal{E}}l_{(i,j)}+ \sum_{i=1}^{m}r_i$, and rewrite~\eqref{eq:probequi} in the following compact form:
\begin{equation} \label{eq:dist}
\minimize\ \mathsf{f}(\mathsf{x})+\mathsf{g}(\mathsf{x})+\tilde{\mathsf{h}}(\mathsf{L}\mathsf{x}),
\end{equation}
 where $\mathsf{f}(\mathsf{x})=\sum_{i=1}^m{f}_i(x_i)$, $\mathsf{g}(\mathsf{x})=\sum_{i=1}^m{g}_i(x_i)$,
 $\tilde{\mathsf{h}}(\mathsf{\tilde{y}},\mathsf{\tilde{w}})=\mathsf{h}(\mathsf{\tilde{y}})+\delta_C(\mathsf{\tilde{w}})$, 
  $\mathsf{h}(\mathsf{\tilde{y}})=\sum_{i=1}^m{h}_i(\tilde{y}_i)$.

In what follows, $\mathcal{S}$ refers to the set of primal-dual solutions of~\eqref{eq:dist}. 
 As in \Cref{subsec:NewPD}, the primal-dual optimality conditions can be written in the form of monotone inclusion~\eqref{eq:inclusion} with 
      \begin{align*}
      A:&(\mathsf{y},\mathsf{w}, \mathsf{x})\mapsto(\partial \mathsf{h}^*(\mathsf{y}),\partial \delta^*_{C}(\mathsf{w}),\partial \mathsf{g}(\mathsf{x})),\\
      M:&(\mathsf{y},\mathsf{w}, \mathsf{x})\mapsto(-L\mathsf{x},-\mathsf{N}\mathsf{x},{L}^\top{\mathsf{y}}+\mathsf{N}^\top \mathsf{w}),\\
      C:&(\mathsf{y},\mathsf{w}, \mathsf{x})\mapsto(0,0,\nabla{} \mathsf{f}(\mathsf{x})),
       \end{align*} 
    where $\mathsf{{u}}=(\mathsf{{y}},\mathsf{{w}})$ represents the dual vector. 

We define the edge weight matrix as follows $$W=\blkdiag\left((\kappa_{(i,j)}I_{2l_{(i,j)}})_{(i,j)\in \mathcal{E}}\right),$$
where the weights $\kappa_{(i,j)}$ are repeated twice (for each of the two neighboring agents). Furthermore, we set
\begin{align*}
\Sigma&= \blkdiag(\sigma_1I_{r_1},\ldots,\sigma_mI_{r_m},W),\\\Gamma&= \blkdiag(\tau_1 I_{n_1},\ldots,\tau_mI_{n_m}),\\
Q&=\blkdiag(\beta_1Q_{1},\ldots,\beta_mQ_{m}).
\end{align*}

Since $\prox_{\tilde{\h}^\star}(\y,\w)=(\prox_{{\mathsf{h}}^\star}(\y),\w-\mathcal{P}_C(\w))$ (using $\prox_{\delta_C}(\cdot) = \mathcal{P}_C(\cdot)$ along with Moreau decomposition \cite[Thm. 14.3]{bauschke2011convex}) the proximal updates of \Cref{Alg:synch-1}, \cf \eqref{eq:opT-nodelay}, become: 
{\mathtight 
  \begin{align*}
  \bar{{y}}_i&=\prox_{\sigma_i{{h}_i}^\star}({y}_i+\sigma_i{L}_i{x}_i),\\
  \bar{{w}}_{(i,j)}\!\!&={w}_{(i,j)}\!+\!\kappa_{(i,j)}({N}_{(i,j)}\mathsf{x}-\mathcal{P}_{C_{(i,j)}}\!\!(\kappa_{(i,j)}^{-1}{w}_{(i,j)}+{N}_{(i,j)}\mathsf{x})),\\
  \bar{{x}}_i&=\prox_{\tau_i{g}_i}\!\!\!({x}_i-\tau_i{L}_i^\top\bar{{y}_i}-\tau_i(\mathsf{N}^\top\bar{\mathsf{w}})_i-\tau_i\nabla {f}({x_i})).
  \end{align*}}
Note that 
for $w_1,w_2\in\R^{l_{(i,j)}}$ the projection onto $C_{(i,j)}$ is
$$\mathcal{P}_{C_{(i,j)}}(w_1,w_2)=\frac{1}{2}\left(w_1-w_2+b_{(i,j)},-w_1+w_2+b_{(i,j)}\right).$$

By assigning to agent $i$ the primal coordinate $x_i$ and dual coordinate $y_i$ and $w_{(i,j),i}$ for all $j\in\mathcal{N}_i$,  \Cref{Alg:dist} is obtained. Note that this assignment entails non-overlapping sets of coordinates, \ie \Cref{ass:4-1} is satisfied.  





The convergence results of \Cref{Alg:dist} are provided separately for the synchronous and asynchronous schemes in the next two theorems, along with a sufficient condition for linear convergence. The proofs follow directly from \Cref{Thm: metricSub,thm:linearCon}.  
\begin{thm}[Convergence of \texorpdfstring{\hyperref[Alg:dist]{Algorithm 3-I}}{Algorithm 3-I}] \label{thm:dist}
	Let \Cref{ass:5,ass:Diststep} hold. 
	  The sequence $\seq{\mathsf{z}^k}=\seq{\mathsf{y}^k,\mathsf{w}^k,\mathsf{x}^k}$ generated by \hyperref[Alg:dist]{Algorithm 3-I} converges to some $z^\star\in\mathcal{S}$. Furthermore, if $f_i$, $g_i$ and $h_i$, $i=1,\ldots,m$ are PLQ, then 
	$\seq{d_{S}(\mathsf{z}^k,\mathcal{S})}$ converges $Q$-linearly to zero, and $\seq{\mathsf{z}^k}$ converges $R$-linearly to $z^\star\in\mathcal{S}$. 
	\end{thm}
\begin{thm}[Convergence of \texorpdfstring{\hyperref[Alg:dist]{Algorithm 3-II}}{Algorithm 3-II}] \label{thm:linearCon-2}
	Let \Cref{ass:5,ass:Diststep} hold. 
	 The sequence $\seq{\mathsf{z}^k}=\seq{\mathsf{y}^k,\mathsf{w}^k,\mathsf{x}^k}$ generated by \hyperref[Alg:dist]{Algorithm 3-II} converges almost surely to some $z^\star\in\mathcal{S}$. Furthermore, if $f_i$, $g_i$ and $h_i$, $i=1,\ldots,m$ are PLQ and  $\seq{\mathsf{z}^k}\subseteq \mathcal{C}$ where $\mathcal{C}$ is a compact set, then $\seq{\E\left[d^2_{\Pi^{-1}S}(\mathsf{z}^{k},\mathcal{S})\right]}$ converges $Q$-linearly to zero. 

	\end{thm}

\section{Application: Formation Control}\label{sec:Formation}

In this section we consider the problem of formation control of a group of robots \cite{raffard2004distributed,schouwenaars2004decentralized}, where each robot/agent has its own local dynamics and cost function and the goal is to achieve a specific formation by communicating only with neighboring agents. 

For simplicity of visualization we consider a $2$D problem. Each subsystem (corresponding to a robot) has four states $x_i=(p_{x_i},p_{y_i},v_{x_i},v_{y_i})$, where $(p_{x_i},p_{y_i})$ and $(v_{x_i},v_{y_i})$ denote the position and the velocity vectors, respectively.  The input for each system is given by $u_i=(v^u_{x_i},v^u_{y_i})$. The discrete-time LTI model of each system is given by 
\begin{equation*}
  x_i(k+1)= \Phi_i x_i(k) + \Delta_i u_i(k), \quad k=0,1,\ldots.
    \end{equation*}  
  The state and input transition matrices are as follows
  \begin{equation*}
    \Phi_i= \begin{pmatrix}
      I &0 & X_1 & 0 \\0 &I &0 & X_1 \\0 &0 &X_2 & 0 \\0 &0 &0 & X_2 
    \end{pmatrix}, \quad \Delta_i= \begin{pmatrix}
      X_3 &0  \\0 &X_3 \\ X_1 & 0 \\0 &  X_1 
    \end{pmatrix}, 
  \end{equation*} 
  where the parameters are $X_1=-t_d(e^{-\tfrac{1}{t_d}}-1)$, $X_2=e^{-\tfrac{1}{t_d}}$ and \mbox{$X_3=t_d^2(e^{-\tfrac{1}{t_d}}-1+ \tfrac{1}{t_d})$} with time constant $t_d = 5$ (s). This discrete-time model was derived from the continuous-time model of \cite{schouwenaars2004decentralized} using exact discretization with step length $\Delta T=1$. 

 Let $N$ denote the horizon length. 
Consider the stacked state and input vectors  $\bm x_i \in \R^{4N},  \bm u_i\in\R^{2N}$: 
 $$\bm x_i \coloneqq (x_i(1),\ldots,x_i(N)),  \;\bm  u_i\coloneqq(u_i(0),\ldots,u_i(N-1)).$$ 
Then the dynamics of each agent can be represented as $\mathcal{A}_i\bm x_i+\mathcal{B}_i\bm u_i=b_i$ where $\mathcal{A}_i$, $\mathcal{B}_i$ are appropriate matrices and $b_i$ depends on the initial state.  The state and input constraints of each agent are represented by the sets $\mathcal{X}_i$, $\mathcal{U}_i$ and are assumed to be easy to project onto, \textit{e.g.}, boxes, halfspaces, norm balls, etc. 
Moreover, we assume that each agent has its own  private objective captured by input and state cost matrices $\mathcal{Q}_i$ and $\mathcal{R}_i$, and vectors $q_i$, $t_i$. 
The specific formation between agents is enforced using another quadratic term that penalizes deviation of two neighbors from the desired relative position.  The optimization problem is described as follows: 
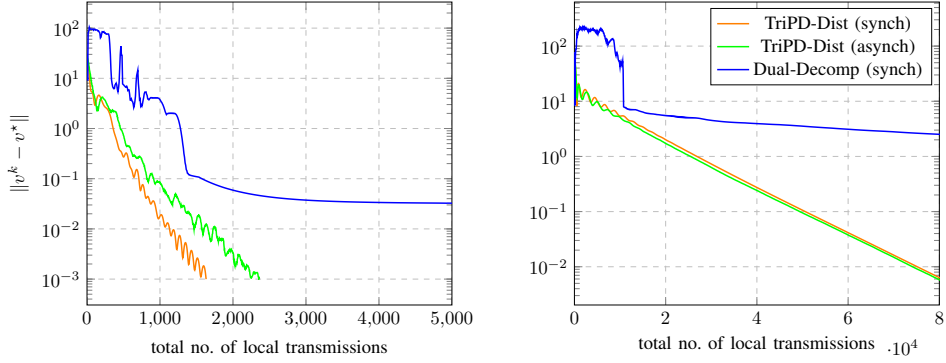
\begin{figure*}
  \center
  \resizebox{0.7\linewidth}{!}{\begin{tikzpicture}
\begin{axis}[ymode=log, name=m5,
xlabel={total no. of local transmissions},
ylabel={$\|v^k-v^\star\|$},
grid=major,
    grid style={line width=.1pt, draw=gray!10, dashed},
    major grid style={line width=.2pt,draw=gray!50},
xmax = 0.5e4,
xmin = 0,
]
\addplot[color=orange,thick,mark=o, mark repeat=180]
file {TeX/Tikz/plotdata/formation/smooth.dat};
\addplot[color=green,thick, mark=square*, mark repeat=400]
file {TeX/Tikz/plotdata/formation/smooth_ra.dat};
\addplot[color=blue,thick, mark=triangle*, mark repeat=300]
file {TeX/Tikz/plotdata/formation/subgrad.dat};

\end{axis}

\begin{axis}[ymode=log, at = (m5.south east), xshift = 2.3cm,
xlabel={total no. of local transmissions},
grid=major,
    grid style={line width=.1pt, draw=gray!10, dashed},
    major grid style={line width=.2pt,draw=gray!50},
xmax = 0.8e5,
xmin = 0,
]
\addplot[color=orange,thick, mark=o, mark repeat=700]
file {TeX/Tikz/plotdata/formation/algsmfifty/smooth.dat};
\addplot[color=green,thick, mark=square*, mark repeat=1000]
file {TeX/Tikz/plotdata/formation/algsmfifty/smooth_ra.dat};
\addplot[color=blue,thick, mark=triangle*, mark repeat=300]
file {TeX/Tikz/plotdata/formation/algsmfifty/subgrad.dat};

	\legend{TriPD-Dist (synch), TriPD-Dist (asynch), Dual-Decomp (synch)}
\end{axis}

\end{tikzpicture}}
  \caption{Comparison for the convergence of the algorithms for $m=5$ (left), and $m=50$ (right). 
  }
  \label{fig1}
\end{figure*}
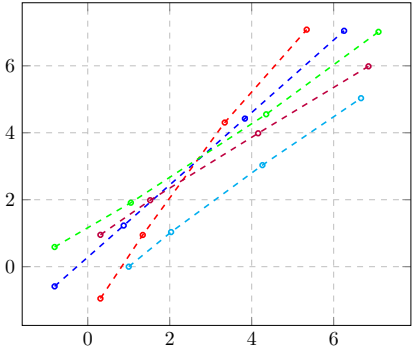
\begin{figure}
  \center
  \resizebox{0.63\linewidth}{!}{	\begin{tikzpicture}
\begin{axis}[grid=major,
    grid style={line width=.1pt, draw=gray!10, dashed},
    major grid style={line width=.2pt,draw=gray!50},
]
\addplot[color=red,thick,dashed,mark size= 1.4pt, mark options={solid},mark=o 
]
file {TeX/Tikz/plotdata/formation/location/positions1.dat};
\addplot[color=blue,thick,mark=o,dashed,mark size= 1.5pt,mark options={solid},mark=pentagon*]
file {TeX/Tikz/plotdata/formation/location/positions2.dat};
\addplot[color=green,thick,mark=o,dashed,mark size= 1.5pt,mark options={solid}, mark=+]
file {TeX/Tikz/plotdata/formation/location/positions3.dat};
\addplot[color=purple,thick,mark=o,dashed,mark size= 1.4pt,mark options={solid},mark=square*]
file {TeX/Tikz/plotdata/formation/location/positions4.dat};
\addplot[color=cyan,thick,mark=o,dashed,mark size= 1.7pt,mark options={solid},mark=diamond*]
file {TeX/Tikz/plotdata/formation/location/positions5.dat};
\end{axis}
	\end{tikzpicture}}
  \caption{Five agents reorganizing from a polygon to an arrow  configuration . 
  }
  \label{fig2}
\end{figure}
	\begin{equation}\label{prob:formation}
    \begin{array}[m]{>{\displaystyle}r >{\displaystyle}l}
    \minimize_{\bm x_i,\bm u_i}{}
    &
  \sum_{i=1}^m\tfrac{1}{2}\|\mathcal{Q}_i\bm x_i- q_i\|^2+\tfrac{1}{2}\|\mathcal{R}_i\bm u_i - t_i\|^2
  \\
  &
  \mathllap{{}+{}}\sum_{i=1}^m\sum_{j\in\mathcal{N}_i}\tfrac{\lambda_i}{2}\|\mathcal{C}(\bm x_i-\bm x_j)-d_{ij}\|^2
  \\[14pt]
  \stt{}
  &
  \mathcal{A}_i\bm x_i+\mathcal{B}_i\bm u_i=b_i,\ 
  \bm x_i\in\mathcal{X}_i,\ 
  \bm u_i\in\mathcal{U}_i
  \\
  &
  i=1,\ldots,m
  \end{array} 
   \end{equation} 
   The relative desired distance of agent $i$ from its neighbor $j$ is given by $d_{ij}$, $\mathcal{C}$ is an appropriate linear mapping that selects the position variables, and $\lambda_i$ is an scalar weight to penalize deviation.    

   For each system that communicates with $i$, \ie, $j\in\mathcal{N}_i$, we introduce a local variable $\bm x_{ij}$, that can be seen as the estimate of $\bm x_j$ kept locally by agent $i$. In order to be consistent hereafter the self variables $\bm x_{i},\bm u_i$ are denoted by $\bm x_{ii},\bm u_{ii}$.  

    For each agent $i=1,\ldots,m$ define the stacked vector
 $${z}_{\mathcal{N}_i}=\left((\bm x_{ij})_{j\in\mathcal{N}_{i}\cup\{i\}},\bm u_{ii}\right)\in\R^{n_i},$$
 where $n_i=4N(|\mathcal{N}_i|+1)+2N$. 

 Let $E_i$ be a linear mapping  such that $E_i z_{\mathcal{N}_i} = \mathcal{A}_i\bm x_{ii}+\mathcal{B}_i\bm u_{ii}$. Hence, the set of points satisfying the dynamics are given by 
 $\mathcal{D}_i=\{z\in\R^{n_i}|E_iz={b}_i\}$.
Consider the linear mapping $L_i$ such that $L_iz_{\mathcal{N}_i}=(\bm x_{ii},\bm u_{ii})$ and denote $\mathcal{Z}_i\coloneqq\mathcal{X}_i\times \mathcal{U}_i$. Moreover, let ${h}_i\coloneqq\delta_{\mathcal{Z}_i}$, ${g}_i\coloneqq\delta_{\mathcal{D}_i}$  and
\begin{align*}
{f}_i(z_{\mathcal{N}_i})\coloneqq&  \tfrac{1}{2}\|\mathcal{Q}_i\bm x_{ii}-q_i\|^2+\tfrac{1}{2}\|\mathcal{R}_i\bm u_{ii}-t_i\|^2\\&+\tfrac{\lambda_i}{2}\textstyle\sum_{j\in\mathcal{N}_i}\|\mathcal{C}(\bm x_{ii}-\bm x_{ij})-d_{ij}\|^2.
\end{align*}



 With these definitions problem \eqref{prob:formation} is cast in the form of problem \eqref{prob:GenProblem} (minimizing over $z_{\mathcal{N}_i}$, $i=1,\ldots,m$) where  the linear mapping $A_{ij}$, for $j\in\mathcal{N}_i$, is such that $A_{ij}z_{\mathcal{N}_i}=(\bm x_{ii},-\bm x_{ij})$ if $i<j$ and $A_{ij}z_{\mathcal{N}_i}=(-\bm x_{ij},\bm x_{ii})$ otherwise. Therefore, we can readily apply \Cref{Alg:dist} to solve the problem in a fully distributed fashion yielding both synchronous and randomized asynchronous implementations.  

In our simulations we used  horizon length $N=3$. For the input and state constraints of all agents we used box constraints: the positions $p_{x_i}$ and $p_{y_i}$ are assumed to be between $0$ and $20$ (m). The velocities $v_{x_i}$ and $v_{y_i}$ and inputs $v_{x_i}^u$ and $v_{y_i}^u$ are assumed to be between between $0$ and $15$ (m/s) (for all agents). 
 The local state cost matrices are set $\mathcal{Q}_i=0.1 I$  for all $i$. The local input cost matrices are set $\mathcal{R}_i=I$ for half of the agents and $\mathcal{R}_i=2I$ for the rest. Moreover, the vectors $q_i$, $t_i$ are set equal to zero, and the penalty parameter $\lambda_i=10$ is used for all the agents.

The stepsizes of \Cref{Alg:dist} were selected as follows: i) (edge stepsizes) $\kappa_{(i,j)}=1$ for all $(i,j)\in\mathcal{E}$, ii) (node stepsizes) $\sigma_i=\beta_i/4$ and $\tau_i=0.99/(\tfrac{\beta_i}{2}+\sigma_i + \sum_{j\in\mathcal{N}_i}\kappa_{(i,j)})$ for all $i$, where we used
$$
\beta_i=\max\{\|\mathcal{Q}_{i}^{\top}\mathcal{Q}_{i}\|+\lambda_i(|\mathcal{N}_{i}|+1),\|\mathcal{R}_{i}^{\top}\mathcal{R}_{i}\|\},
$$
which is an upper bound for the Lipschitz constant of $\nabla f_i$. It is plain to see that the above choice of stepsizes for the agents satisfy~\Cref{cond:cord}. Note that the stepsize selection only requires local parameters $\mathcal{R}_i$, $\mathcal{Q}_i$, $\lambda_i$ and the number of neighbors $|\mathcal{N}_i|$, \ie, the algorithm can be implemented without \emph{any} global coordination.

In our simulations, we considered $m$ robots initially in a polygon configuration and enforced an arrow formation by appropriate selection of $d_{ij}$ in \eqref{prob:formation}.  This scenario is depicted for $m=5$ in \Cref{fig2}. The neighborhood relation in this case is taken to be the same  arrow configuration, \ie, all agents have two neighbors apart from two agents with only one  neighbor. 

For comparison we considered the dual decomposition approach of \cite{raffard2004distributed} (based on the subgradient method). Notice that dual decomposition with gradient or accelerated gradient methods can not be applied to this problem since ${f}_i$'s are convex but not strongly convex.
Recently, \Cref{Alg:dist} was compared against the  dual accelerated proximal gradient method, in the context of distributed model predictive control (with strongly convex quadratic cost) \cite{Latafat2018EC}. 


In the simulations for \Cref{fig1}, we used the stepsize $10/k$ (as tuned for achieving better performance) for the dual decomposition method where $k$ is the number of iterations.
 Notice that the dual decomposition approach for this problem can not achieve a full splitting of the operators involved: at every iteration agents need to solve an inner minimization (we used MATLAB's \texttt{quadprog} to perform this step), the result of which must be communicated to the neighbors for their computation, and is followed by another communication round. This extra need for synchronization would further slow down the algorithm in practical implementations \cite{freris2011fundamental}.

\Cref{fig1} demonstrates the superior performance 
 of both the synchronous and asynchronous versions of \Cref{Alg:dist} compared to the dual decomposition approach. The $y$-axis is the distance of $v^k\coloneqq(\bm x_{11}^k,\bm u_{11}^k,\ldots, \bm x_{mm}^k,\bm u_{mm}^k)$ from the solution ($v^\star$ was computed by solving \eqref{prob:formation} in a centralized fashion). The $x$-axis denotes the total number of local transmissions between agents. In the asynchronous implementation we used independent activation probabilities $p_i=0.5$ for all agents. It is observed that the total number of local iterations is similar to that of the synchronous implementation.   
Finally, as evident in \Cref{fig1} both versions of \Cref{Alg:dist} achieve linear convergence rate as predicted by \Cref{thm:dist,thm:linearCon-2} (the functions $f_i,g_i$ and $h_i$ are PLQ).

   \section{Conclusions}
   \label{sec:conclusions}
The primal-dual algorithm introduced in this paper enjoys several structural properties that distinguish it from other related methods in the literature. A key property, that has  been instrumental in developing a block-coordinate version of the algorithm, is the fact that 
the generated sequence is $S$-Fej\'er monotone, where $S$ is a block diagonal positive definite matrix.
 It is shown that the algorithm attains linear convergence under a metric subregularity assumption that holds for a wide range of cost functions that are not necessarily strongly convex. 
The block-coordinate version of the developed algorithm is exploited to devise a novel fully distributed asynchronous method for multi-agent optimization over graphs. 
Our future work includes designing a block-coordinate version of the \emph{SuperMann} scheme of \cite{themelis2016supermann} 
that applies to quasi-nonexpansive operators. 
In light of the fact that this method enjoys superlinear convergence rates, such extension is especially attractive for multi-agent optimization 
yielding schemes with faster convergence and fewer communication rounds. 
Other research directions enlist investigating extensions to account for directed and time-varying topologies, communication delays, and designing efficient strategies for selecting activation probabilities and stepsizes.


\let\appendix\appendices    
\let\endappendix\endappendices

\begin{appendix}
	\crefalias{section}{appsec}
	\phantomsection
	\section{}
	\label{sec:appendix}

\begin{appendixproof}{lem:nablaf-dg}
	  	Consider the operator  $T$ as in~\eqref{eq:op2}. By monotonicity of $A$ at $z^\star$ and $\bar{z}$ along with~\eqref{eq:op1} we have 
  	\begin{equation} \label{eq:-1p}
  	0\leq \langle -Mz^\star-Cz^\star+Mz+Cz-Hz+H\bar{z}, z^\star - \bar{z}\rangle.
  	\end{equation}
    For $\beta_f>0$, \cref{ass:1-3} is equivalent to $\nabla f$ being \emph{cocoercive} \cite[Thm. 18.16]{bauschke2011convex}, \ie, for all $x,y\in\R^n$:
      \begin{equation}\label{eq:Coco}
      \tfrac{1}{\beta_f}\|\nabla f(x)-\nabla f(y) \|_{Q^{-1}}^2 \leq \langle\nabla f(x)-\nabla f(y), x-y\rangle. 
      \end{equation}
  	On the other hand, for $\beta_f>0$ we have 
       \begingroup
\allowdisplaybreaks
	  	\begin{align}\langle Cz-Cz^{\star}&  ,z^{\star}-\bar{z}\rangle=\langle\nabla f(x)-\nabla f({x}^{\star}),x^{\star}-\bar{x}\rangle\nonumber\\
  		 ={}&\langle\nabla f(x)-\nabla f({x}^{\star}),x-\bar{x}\rangle\nonumber\\
  		& +\langle\nabla f(x)-\nabla f({x}^{\star}),x^{\star}-x\rangle\nonumber\\
  		 \leq{}&\tfrac{1}{\beta_f}\|\nabla f(x)-\nabla f(x^{\star})\|_{Q^{-1}}^{2}+\tfrac{\beta_f}{4}\|x-\bar{x}\|_{Q}^{2}\nonumber\\
  		& +\langle\nabla f(x)-\nabla f({x}^{\star}),x^{\star}-x\rangle\nonumber\\
  		 \leq{}&\langle\nabla f(x)-\nabla f({x}^{\star}),x-x^{\star}\rangle+\tfrac{\beta_f}{4}\|x-\bar{x}\|_{Q}^{2}\nonumber\\
  		& +\langle\nabla f(x)-\nabla f({x}^{\star}),x^{\star}-x\rangle,\nonumber\\
  		={}&\tfrac{\beta_f}{4}\|x-\bar{x}\|_{Q}^{2},\label{eq:coco2}
  		\end{align}
      \endgroup 
  	where we have used \eqref{eq:Fenchel-Young} (with  $V=\tfrac{2}{\beta_f}Q^{-1}$) in the first inequality,  and~\eqref{eq:Coco} in the second inequality, respectively. Notice that if $\beta_f=0$ then inequality \eqref{eq:coco2} holds trivially with equality.     
  	
  	 Using~\eqref{eq:coco2} in~\eqref{eq:-1p}, along with skew-symmetry of $K$ and $M$, we have	
   \begingroup
\allowdisplaybreaks
  		\begin{align}
  		0\leq& \langle -Mz^\star-Cz^\star+Mz+Cz-Hz+H\bar{z}, z^\star - \bar{z}\rangle  \nonumber\\ 
  		\leq& \langle(M-K)({z}-z^\star)+P(\bar{z}-{z}), z^\star - \bar{z}\rangle+\tfrac{\beta_f}{4}\|{x}-\bar{{x}}\|_Q^{2} \nonumber\\ 
  		=&   \langle(M-K)({z}-z^\star)+P(\bar{z}-{z}), z^\star - {z}\rangle + \tfrac{\beta_f}{4}\|{x}-\bar{{x}}\|_Q^{2}\nonumber\\
  		&+\langle(M-K)({z}-z^\star)+P(\bar{z}-{z}), {z} - \bar{z}\rangle \nonumber\\
  		=&  \langle P(\bar{z}-{z}), z^\star - {z}\rangle + \tfrac{\beta_f}{4}\|{x}-\bar{{x}}\|_Q^{2}-\|\bar{z}-{z}\|_P^2\nonumber\\
  		&+\langle(M-K)({z}-z^\star), {z} - \bar{z}\rangle 
  		\nonumber\\
  		=&  \langle {z}-z^\star, (H+M^\top)({z}- \bar{z})\rangle \nonumber \\&+\tfrac{\beta_f}{4}\|{x}-\bar{x}\|_Q^{2} - \|\bar{z}-{z}\|_P^2. \label{eq:-12}
  		\end{align}
      \endgroup
 By definition, $S^{-1}(H+M^\top)(\bar{z}-z)=Tz-z$. Thus
  	\begin{align}
  	\langle {z}-z^\star, (H+M^\top)({z}- \bar{z})\rangle &= \langle {z}-z^\star, {z}- Tz\rangle_S.
  	\label{eq:-6-dg}
  	\end{align}
  	On the other hand, we have $\bar{z}-z=(H+M^\top)^{-1}S(Tz-z)$. 
  	Using \eqref{eq:PD-dg},~\eqref{eq:HplusM} and~\eqref{eq:Tz} we conclude 
  	\begin{equation}\label{eq:-7-dg}
  	\|\bar{z}-z\|_{P}^{2}-\tfrac{\beta_f}{4}\|\bar{{x}}-{x}\|_Q^{2}=  \|Tz-z\|_{\tilde{P}}^{2},  
  	\end{equation}
  	where $\tilde{P}$ is defined in~\eqref{eq:tildeP}. 
  	Combining~\eqref{eq:-12}, \eqref{eq:-6-dg} and~\eqref{eq:-7-dg} completes the proof. 
\end{appendixproof}
\begin{appendixproof}{thm:synch}
  We establish convergence by showing that the  sequence $\seq{z^k}$ is Fej\'{e}r monotone with respect to $\mathcal{S}=\fix T$. 
    We have
  	\begin{align}
  	\|z^{k+1}-z^\star\|_{S}^2 =& \|Tz^k-z^k+z^k-z^\star\|_{ S}^2\nonumber\\ 
  	={}&\|z^k-z^\star\|_{S}^2 +  \|Tz^k-z^k\|_{S}^2 \nonumber\\&+2\langle z^k-z^\star,Tz^k-z^k \rangle_S\nonumber\\
  	\leq{}&\|z^k-z^\star\|_{S}^2 -  \|Tz^k-z^k\|_{2{\tilde{P}}-S}^2, \label{Fej:synch}
  	\end{align}
  	where the inequality follows from~\Cref{lem:nablaf-dg}. Note that $2\tilde{P}- S$
     is symmetric positive-definite if and only if~\cref{ass:2} holds. Therefore, by~\eqref{Fej:synch} the sequence $\seq{z^k}$ is Fej\'{e}r monotone in the space equipped with  inner product $\langle \cdot,\cdot\rangle_S$; in particular, $\seq{z^k}$ is bounded.  Furthermore, it follows from~\eqref{Fej:synch} and the fact that $2\tilde{P}-S$ is positive-definite that 
  	\begin{equation}\label{eq:Tz-z}
  	\|Tz^k-z^k\|\rightarrow 0.
  	\end{equation} 
  	 The operator $T$ is continuous (since it involves proximal and linear mappings that are continuous, and since $\nabla f$ is assumed continuous).
  	 Let $z^c$ be a cluster point of $\seq{z^k}$. It follows from the continuity of $T$ and~\eqref{eq:Tz-z} that $Tz^c-z^c=0$, \ie, $z^c\in\fix T$. 
  	 The result follows from Fej\'er monotonicity of $\seq{z^k}$ with respect to $\mathcal{S}=\fix T$ and \cite[Thm. 5.5]{bauschke2011convex}. 
\end{appendixproof}
\begin{appendixproof}{thm:Fejerlike-dg} 
Let us define the operator $E^k\coloneqq \sum_{i=1}^m \epsilon^k_i U_i$ that maps the elements of $(\R^{n+r},\mathcal{F}_{k-1})$ to $(\R^{n+r},\mathcal{F}_{k})$. The iterations of \Cref{Alg:BC} can be written as $z^{k+1}=z^k+E^{k+1}(Tz^k-z^k)$. 
 We have 
  \begingroup
\allowdisplaybreaks
\begin{align}
  \mathbb{E}_{k} \circ& E^{k+1} = \sum_{\varepsilon\in\Psi}\mathbb{P}(\epsilon^{k+1}=\varepsilon)\sum_{j=1}^{m}\varepsilon_{j}U_{j}\nonumber\\&=\sum_{j=1}^{m}\sum_{\varepsilon\in\Psi}\mathbb{P}(\epsilon^{k+1}=\varepsilon) \varepsilon_{j}U_{j}\nonumber\\
  &=\sum_{j=1}^{m}\sum_{\varepsilon\in\Psi,\varepsilon_{j}=1}\mathbb{P}(\epsilon^{k+1}=\varepsilon)U_j = \sum_{j=1}^{m}p_{j} U_j = \Pi,\label{eq:circE} 
  \end{align}
  \endgroup
  where we used \Cref{ass:4-2,ass:4-1}. 
%
Therefore, we have 
 \begingroup
\allowdisplaybreaks
\begin{align}
  \mathbb{E}_{k}&\left[\|z^{k+1}-z^{\star}\|_{\Pi^{-1}S}^{2}\right]
  &
   \nonumber\\
   ={}&
  \mathbb{E}_{k}\left[\|z^{k}+ E^{k+1}(Tz^k-{{z}^k})-z^{\star}\|_{\Pi^{-1}S}^{2}\right]\nonumber\\
={}&\|z^{k}-z^{\star}\|_{\Pi^{-1}S}^{2}+2\langle z^{k}-z^{\star}, \mathbb{E}_{k} \left[E^{k+1}(Tz^k-{{z}^k})\right]\rangle_{\Pi^{-1}S}\nonumber\\
&+\mathbb{E}_{k}\left[\langle E^{k+1}(Tz^{k}-z^{k}),E^{k+1}(Tz^{k}-z^{k})\rangle_{\Pi^{-1}S} \right]\nonumber\\  
={}&\|z^{k}-z^{\star}\|_{\Pi^{-1}S}^{2} +\|Tz^{k}-z^{k}\|_{S}^{2} \nonumber\\
& +2\langle z^{k}-z^{\star},Tz^{k}-z^{k}\rangle_{S} \nonumber
\end{align}
\endgroup
where we used \eqref{eq:circE} and the fact $E^k$ is self-adjoint and idempotent (since $U_i$ are $0$-$1$ matrices) in the last equality. Inequality \eqref{eq:stoFej} follows by using \eqref{eq:main-innerprod}. 
The convergence of the sequence follows from \eqref{eq:stoFej} 
using the Robbins-Siegmund lemma \cite{robbins1985convergence} 
and arguing as in \cite[Thm. 3]{bianchi2015coordinate} and \cite[Prop. 2.3]{combettes2015stochastic}.
	\end{appendixproof}
\begin{appendixproof}{lem:QuadGrow}
  From the equivalent characterization of strong subregularity in \eqref{eq:quadgrowth-s} we have that there exists a neighborhood $\mathcal{U}_{{x}^\star}$ of ${x}^\star$ such that for all $x\in\mathcal{U}_{x^\star}$
  \begin{align}
  (f+g)(x)\geq&(f+g)(x^{\star})+\langle -L^{\top}u^{\star},x-x^{\star}\rangle\nonumber\\&+c_{1}\|x-x^\star\|^2,\label{eq:gPlusg}
  \end{align}
  and a neighborhood $\mathcal{U}_{{u}^\star}$ of ${u}^\star$ such that for all $u\in\mathcal{U}_{u^\star}$
  \begin{equation*}
  \stepcounter{equation}\tag{\arabic{equation}}
  \label{eq:hquad}
  h^{*}(u)\geq h^{*}(u^{\star})+\langle Lx^{\star},u-u^{\star}\rangle\nonumber+c_{2}\|u-u^\star\|^2.  
  \end{equation*}
  Fix $z=(u,x)$ with $u\in\mathcal{U}_{u^\star}$ and $x\in\mathcal{U}_{x^\star}$. Consider $v=(v_1,v_2)\in Fz\coloneqq Az+Mz+Cz$. By definition (cf.~\eqref{monoper}) we have%
   \begin{equation*} 
   \begin{cases}
   v_{1}  \in\partial h^{*}(u)-Lx, & \ \ \ \ \ \ \\ v_{2} \in\partial g(x)+ \nabla f(x)+L^{\top}u.& \ \ \ \ \ \ 
   \end{cases} 
   \end{equation*} 
  Using this together with the definition of subdifferential yields: 
  \begin{align}
  \langle v_{1}+Lx,u-{u}^\star\rangle&\geq h^{*}(u)-h^{*}({u}^\star),\label{eq:in1}\\
  \langle v_{2}-L^{\top}u,x-x^\star\rangle&\geq (f+g)(x)-(f+g)(x^\star). \label{eq:in2}
  \end{align}
  Combining \eqref{eq:in1}, \eqref{eq:in2} with \eqref{eq:gPlusg}, \eqref{eq:hquad} and noting that
  $$\langle L^\top(u^\star-u),x-x^\star\rangle + \langle L(x-x^\star),u-u^\star\rangle = 0,$$
  yields: 
  \begin{align*}
  \langle v,z-{z}^\star\rangle&=\langle v_1,u-{u}^\star\rangle+ \langle v_2,x-{x}^\star\rangle\\ & \geq c_{2}\|u-{u}^\star\|^{2}+c_{1}\|x-{x}^\star\|^{2} \geq c\|z-z^\star\|^2, 
  \end{align*}
  where $c=\min\{c_1,c_2\}$. Therefore, by the Cauchy-Schwarz inequality 
  $\|v\|\geq c\|z-z^\star\|$. Since $\|z-z^\star\|\geq d(z,F^{-1}0)$, and $v\in Fz$ was selected arbitrarily, we have 
  \begin{equation}
  d(z,F^{-1}0)\leq \tfrac{1}{c}d(0,Fz)\quad \forall z\in {\mathcal{U}_{u^\star}}\times{\mathcal{U}_{x^\star}}. \label{eq:subreqpro}
  \end{equation}
 Thus $F$ is metrically subregular at $z^\star$ for $0$. 
  
  To establish uniqueness of the primal-dual solution consider:
    \begin{equation*}
    \mathcal{L}(u,x)\coloneqq (f+g)(x)+\langle Lx,u\rangle -h^*(u).
    \end{equation*}
     Adding~\eqref{eq:gPlusg} and \eqref{eq:hquad} yields 
    \begin{equation}\label{eq:Lagrangian}
    \mathcal{L}(u^\star,x)-   \mathcal{L}(u,x^\star)\geq c\|z-z^\star\|^2 \quad \forall z\in {\mathcal{U}_{u^\star}}\times{\mathcal{U}_{x^\star}}
    \end{equation}  
    Let $\bar{z}^\star=(\bar{u}^\star,\bar{x}^\star)\in\mathcal{S}$ such that  $\bar{z}^\star\in \mathcal{U}_{u^\star}\times \mathcal{U}_{x^\star}$. Since $\bar{z}^\star$ is also a primal-dual solution we have
$\mathcal{L}(\bar{u}^\star,x^\star)-    \mathcal{L}(u^\star,\bar{x}^\star)\geq 0$.    Therefore, using \eqref{eq:Lagrangian} at $\bar{z}^\star$ yields $\bar{z}^\star=z^\star$. Since $\mathcal{S}$ is convex, we conclude that it is a singleton, \ie, $\mathcal{S}=\{z^\star\}$. Consequently it follows from \eqref{eq:subreqpro} that $F$ is strongly subregular at $z^\star$ for $0$. 

The second part is a direct consequence of the first part and the fact that if \Cref{ass:Quad} holds globally then also the quadratic growth conditions \eqref{eq:gPlusg} and \eqref{eq:hquad} hold globally, \ie, $\mathcal{U}_{x^\star}=\R^n$, $\mathcal{U}_{u^\star}\in\R^r$. This can be shown by adapting the proof of \cite[Thm. 3.3]{aragon2008characterization}.
  \end{appendixproof}

\begin{appendixproof}{lem:PLQ}
Since $f$, $g$ and $h$ are proper closed convex PLQ,  
the subdifferentials $\partial g$, $\nabla f$ and $\partial h^*$ are  piecewise polyhedral mappings \cite[
Prop. 12.30(b), Thm. 11.14(b)]{rockafellar2009variational}.  
 The graph of $M$ is polyhedral, since $M$ is linear. Therefore, the sum $F=A+M+C$ is also piecewise polyhedral. 
Since the inverse of a piecewise polyhedral mapping is piecewise polyhedral, the result follows from \cite[3H.1 and 3H.3]{dontchev2009implicit}. 
  \end{appendixproof}

\begin{appendixproof}{thm:linearCon}%
	For notational convenience let $\bar{S}=\Pi^{-1}S$ and  note that $\mathcal{S}=\zer F$ (\cf \eqref{eq:SfixT}). 
	By definition we have $\|{z}^k-\mathcal{P}^{\bar{S}}_{\mathcal{S}}({z}^k)\|_{\bar{S}}=d_{\bar{S}}({z}^k,\mathcal{S})$ (where the minimum is attained since $\mathcal{S}$ is a closed convex set). Consequently, it follows from~\eqref{eq:stoFej} 	that
	{\small
	\begin{align}
	{\E}_k\left[d^2_{\bar{S}}({z}^{k+1},\mathcal{S})\right] & \leq {\E}_k\left[\|z^{k+1}-\mathcal{P}^{\bar{S}}_{\mathcal{S}}(z^k)\|^2_{\bar{S}}\right] \nonumber \\
	& \leq \|z^{k}-\mathcal{P}^{\bar{S}}_{\mathcal{S}}(z^k)\|_{\bar{S}}^{2} -\|T{z}^k-{z}^{k}\|_{2\tilde{P}-S}^{2}\nonumber\\&= d_{\bar{S}}^2(z^k,\mathcal{S})  -\|T{z}^k-{z}^{k}\|_{2\tilde{P}-S}^{2}. \label{eq:dlineart}
	\end{align}}
	By definition~\eqref{eq:op2}, we have 
	{\small \begin{align}
	&\|\bar{z}^k-z^k\|^2= \|(H+M^\top)^{-1}S(Tz^k-z^k)\|^2 \nonumber\\ & \leq  \|(H+M^\top)^{-1}S\|^2\|{(2\tilde{P}-S)}^{-1}\|\|Tz^k-z^k\|^2_{2\tilde{P}-S},  \label{eq:ztz}
	\end{align}}%
	where $\bar{z}^k$ is defined by~\eqref{eq:op1} applied at $z=z^k$. Consider the projection of $\bar{z}^k$ onto $\mathcal{S}$, $\mathcal{P}_{\mathcal{S}}(\bar{z}^k)$. By definition  $\|\bar{z}^k-\mathcal{P}_{\mathcal{S}}(\bar{z}^k)\|=d(\bar{z}^k,\mathcal{S})$, 
	and we have
	\begin{align}
	d_{\bar{S}}^2(z^k,\mathcal{S})  &\leq \|z^k-\mathcal{P}_{\mathcal{S}}(\bar{z}^k)\|_{\bar{S}}^2 \leq \|{\bar{S}}\|\|z^k-\mathcal{P}_{\mathcal{S}}(\bar{z}^k)\|^2 \nonumber\\& \leq  \|{\bar{S}}\|\left(\|\bar{z}^k-\mathcal{P}_{\mathcal{S}}(\bar{z}^k)\|+\|\bar{z}^k-z^k\|\right)^2\nonumber\\& = \|{\bar{S}}\|\left(d(\bar{z}^k,\mathcal{S})+\|\bar{z}^k-z^k\|\right)^2.  
	\label{eq:-40}
	\end{align}
In what follows we bound $d(\bar{z}^k,\mathcal{S})$ by $\|\bar{z}^k-z^k\|$. 
 Define 
    \begin{equation}\label{eq:wn}
  v^k\coloneqq -(H-M)(\bar{z}^k-z^k) + C\bar{z}^k- Cz^k.
  \end{equation}
  It follows from~\eqref{eq:op1} that $(H-M-C)z^k\in(H+D)\bar{z}^k$, which in turn implies 
  \begin{equation*}
  v^k\in F\bar{z}^k=(A+M+C)\bar{z}^k.
  \end{equation*} 
  Consequently, using (global) metric subregularity of $F$ yields 
  \begin{equation}
  d(\bar{z}^{k},\mathcal{S})\leq\eta\|v^k\|. \label{eq:metricSub}
  \end{equation}
   By the triangle inequality and Lipschitz continuity of $C$,
  {\mathtight\begin{align}
  &\|v^k\|  =\|(H-M)(\bar{z}^k-z^k)-C\bar{z}^k+C{z^k}\| \nonumber\\&\leq \|(H-M)(\bar{z}^k-z^k)\|+\|C\bar{z}^k-Cz^k\| \leq\xi\|\bar{z}^k-z^k\|, \label{eq:vbound}
  \end{align}}
  where $\xi=\|H\!-\!M\| + \beta_f\|Q\|$.  
  By \eqref{eq:metricSub} and \eqref{eq:vbound} we have
  \begin{equation*}
  d(\bar{z}^{k},\mathcal{S})\leq \xi\eta \|\bar{z}^k-z^k\|.
  \end{equation*}
   Combine this with~\eqref{eq:ztz} and \eqref{eq:-40} to derive 
  \begin{align} \label{eq:-1a}
  d_{\bar{S}}^2(z^k,\mathcal{S})  \leq \phi \|Tz^k-z^k\|^2_{2\tilde{P}-S},
  \end{align}
  where $\phi = (\xi\eta+1)^2\|(H+M^\top)^{-1}S\|^2\|{(2\tilde{P}-S)}^{-1}\|\|\bar{S}\|$. Therefore, by~\eqref{eq:dlineart} and~\eqref{eq:-1a} we have
  \begin{align*}
  {\E}_k\left[d^2_{\bar{S}}({z}^{k+1},\mathcal{S})\right] \leq  d_{\bar{S}}^2(z^k,\mathcal{S})  -\tfrac{1}{\phi} d_{\bar{S}}^2(z^k,\mathcal{S}). 
  \end{align*}
  Taking expectation in both sides concludes the proof.
  For the case of PLQ functions, let $\mathcal{U}_{z_\star}$ denote an open subregularity neighborhood around $z_\star\in\mathcal{S}$, and set $\mathcal{U}_{\star}\coloneqq \cup_{z^\star\in\mathcal{S}}\mathcal{U}_{z_\star}$. By \Cref{lem:PLQ} there exists a positive $\eta$ such that $d(z,F^{-1} 0) \leq \eta d(0,Fz)$ for $z\in \mathcal{U}_{\star}$. Moreover, since $\seq{z^k}\subseteq\mathcal{C}$ up to possibly enlarging $\mathcal{C}$ we have $\seq{\bar z^k}\subseteq\mathcal{C}$. 
  Note that since $\seq{z^k} \subseteq \mathcal{C}$ and $\mathcal{C}$ is closed, $\mathcal{C}\cap \mathcal{S} \neq \emptyset$ and $\mathcal{C}\cap \mathcal{U}_\star \neq \emptyset$. 
It is sufficient to show that $d(z,F^{-1}0)\leq \eta' d(0,Fz)$ for $z\in\mathcal{C}$. 
    Let us define $D(z) \coloneqq d(0,Fz)$.  Since $\gra F$ is closed, $D(z)$ is lower semicontinuous \cite[Thm. 5.7, Prop. 5.11(a)]{rockafellar2009variational}. 
By \cite[Cor. 1.10]{rockafellar2009variational} $D(z)$ attains a minimum over the compact set $\mathcal{C}\setminus \mathcal{U}_{\star}$: $c_{d}\coloneqq \min_{z\in\mathcal{C}\setminus \mathcal{U}_\star} D(z)>0$ where the strict inequality is due to the fact that the minimizer belongs to $\mathcal{C}\setminus \mathcal{U}_\star$. 
Moreover, $c_\mathcal{C}\coloneqq \sup_{z\in\mathcal{C}} d(z,F^{-1}0)<\infty$ due to the fact that  $\mathcal{C}$ is bounded. Hence  
  $d(z,F^{-1} 0) \leq c_\mathcal{C} \leq  \tfrac{c_\mathcal{C}}{c_d} d(0,Fz)$  for $z\in \mathcal{C}\setminus \mathcal{U}_{\star}$.
Therefore, by combining the two cases we obtain $d(z,F^{-1} 0) \leq \max\{ \tfrac{c_\mathcal{C}}{c_d}, \eta \} d(0,Fz)$ for $ z\in \mathcal{C}$ as claimed.  
  The second sufficient condition follows from \Cref{lem:QuadGrow}. 
  \end{appendixproof}

\begin{appendixproof}{lem:EquivMetr}
  First we show the if statement: assume that $R=\id-T$ is metrically subregular at $z^\star$ for $0$. Then there exists $\eta>0$ and a neighborhood $\mathcal{U}$ of $z^\star$ such that 
    \begin{equation}\label{eq:110}
      d(z,R^{-1}0) \leq \eta d(0,Rz)\quad \forall z\in\mathcal{U}.
    \end{equation}
    The two sets $R^{-1}0$ and $F^{-1}0$ are equal, \cf \eqref{eq:SfixT}. In what follows, we upper bound $d(0,Rz)$ by $d(0,Fz)$. 
    Let $w\in Fz = A{z}+M{z}+C{z}$. By  \eqref{eq:op1} we have that
    \begin{equation*}
      Hz-Mz-Cz-H\bar{z}\in A\bar{z}. 
    \end{equation*}
    Using this together with the monotonicity of $A$ at $z$ and $\bar{z}$, we obtain:
    \begin{align*}
      0\leq&\langle z-\bar{z},\left(w-M{z}-Cz\right)-\left(Hz-Mz-Cz-H\bar{z} \right)\rangle\nonumber\\=&  \langle z-\bar{z},w-Hz+H\bar{z}\rangle = \langle z-\bar{z},w\rangle - \|\bar{z}-z\|_P^2, 
    \end{align*}
    where in the last equality we have used the fact that $H=P+K$ and $K$ is skew-symmetric. 

   By the Cauchy–Schwarz inequality 
       \begin{equation*}
         \|\bar{z}-z\|_P^2\leq \langle z-\bar{z},w\rangle \leq \|\bar{z}-z\|_P \|w\|_{P^{-1}}, 
       \end{equation*}
       therefore
       \begin{equation}\label{eq:1101}
         \|\bar{z}-z\|_P \leq \|w\|_{P^{-1}}. 
       \end{equation}
       On the other hand by \eqref{eq:op2}:
       \begin{equation*}
         \|Rz\| \leq \|S^{-1}(H+M^\top)P^{-1/2}\|\|\bar{z}-z\|_P. 
       \end{equation*}
       Combine this with \eqref{eq:110} and \eqref{eq:1101} to obtain
       \begin{align*}
         d(z,F^{-1}0)=&d(z,R^{-1}0) \leq \eta \|Rz\|\nonumber \\\leq& \eta \|S^{-1}(H+M^\top)P^{-1/2}\|\|P^{-1}\|^{1/2}\|w\|.
       \end{align*}
       Since $w\in Fz$ was arbitrary, we conclude that $F$ is metrically subregular at $z^\star$ for $0$  (with a different subregularity modulus). 
   
       Next we prove the \emph{only if} statement: assume that $F$ is metrically subregular at $z^\star$ for $0$, \ie, there exists $\eta>0$ and neighborhood $\mathcal{U}$ of $z^\star$ such that 
         \begin{equation}\label{eq:110b}
      d(z,F^{-1}0) \leq \eta d(0,Fz) \quad \forall z\in\mathcal{U}.
    \end{equation}
    By \eqref{eq:-12} and the Cauchy–Schwarz inequality we infer that
    \begin{equation*}
      \|\bar{z}-z\|\leq c \|z-z^\star\|,
    \end{equation*}
    for some positive constant $c$. Hence, there exists a neighborhood $\bar{\mathcal{U}}\subset\mathcal{U}$ of $z^\star$ such that if $z\in\bar{\mathcal{U}}$ then $\bar{z}\in\mathcal{U}$. Fix a point $z\in\bar{\mathcal{U}}$ so that $\bar{z}\in\mathcal{U}$. By \eqref{eq:110b} it holds that:   
    \begin{equation}\label{eq:110d}
      d(\bar{z},F^{-1}0) \leq \eta d(0,F\bar{z}).
    \end{equation}
    Define $v$ as in \eqref{eq:wn} (dropping the iteration index $k$). Noting that $v\in F\bar{z}$, it follows from \eqref{eq:110d} that 
   \begin{equation}\label{eq:110c}
      d(\bar{z},F^{-1}0) \leq \eta \|v\| \leq \eta \xi\|\bar{z}-z\|,
    \end{equation}
    where we used \eqref{eq:vbound} in the second inequality.
    Invoking triangle inequality we have 
    \begin{align}
      d({z},{R}^{-1}0) = & d({z},F^{-1}0) \leq d(\bar{z},F^{-1}0) + \|\bar{z}-z\|\nonumber \\ \leq & (1+\eta \xi)\|\bar{z}-z\|.\label{eq:1209} 
    \end{align}
    On the other hand by \eqref{eq:op2} it holds that 
    \begin{equation*}
    \|\bar{z}-z\|\leq \|(H+M^\top)^{-1}S\|\|Rz\|.
    \end{equation*}
    Combining this with \eqref{eq:1209} yields  
    \begin{align*}
      d({z},{R}^{-1}0)  \leq  (1+\eta \xi)\|(H+M^\top)^{-1}S\|\|Rz\| \quad \forall z\in\bar{\mathcal{U}},
    \end{align*}
\ie that $R$ is metrically subregular at $z^\star$ for $0$.
   \end{appendixproof}

\end{appendix}

\bibliographystyle{IEEEtran}
\bibliography{references}

 \begin{IEEEbiography}[{\includegraphics[width=1in,height=1.25in,clip,keepaspectratio]{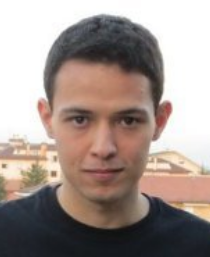}}]{Puya Latafat}
 	is currently working towards a joint PhD at the Department of Electrical Engineering (ESAT) of KU Leuven (Belgium) and IMT School for Advanced Studies Lucca (Italy). He received his M.Sc. in Mathematical Engineering jointly from University of L'Aquila (Italy) and University of Hamburg (Germany), and his B.Sc. in Electrical Engineering from University of Tabriz (Iran).  His research interests revolve around large-scale and distributed optimization with applications to model predictive control and machine learning. 
 \end{IEEEbiography}

 \begin{IEEEbiography}[{\includegraphics[width=1in,height=1.25in,clip,keepaspectratio]{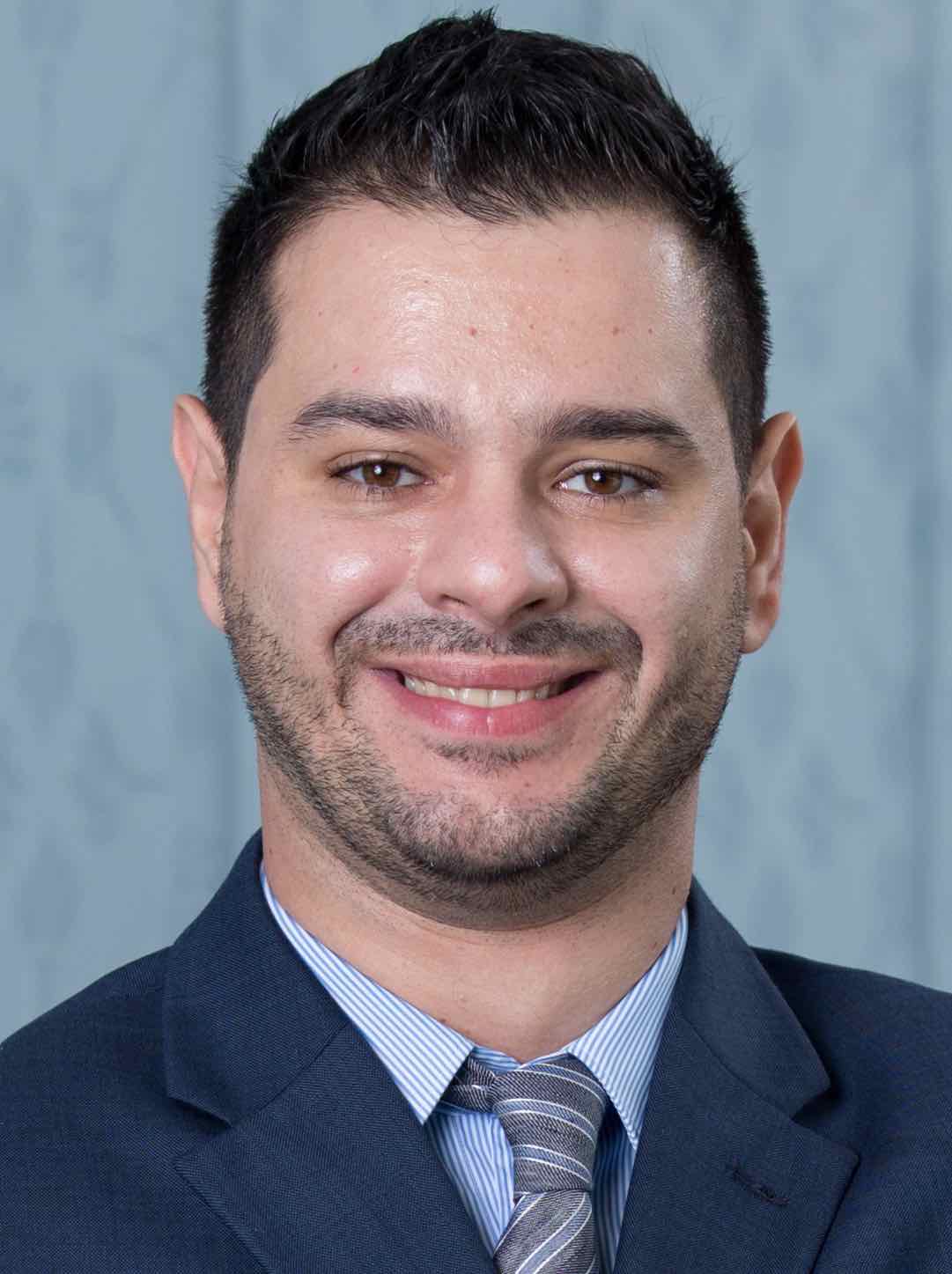}}]{Nikolaos M. Freris}
	is Professor with the School of Computer Science and Technology at the University of Science and Technology of China (USTC). He received a Diploma in Electrical and Computer Engineering from the National Technical University of Athens, Greece in 2005, an M.S. degree in Electrical and Computer Engineering, an M.S. degree in Mathematics, and a Ph.D. degree in Electrical and Computer Engineering all from the University of Illinois at Urbana-Champaign in 2007, 2008, and 2010, respectively. Dr. Freris's research interests lie in the area of cyberphysical systems: distributed estimation, optimization, and control, machine learning, wireless networks, signal processing, and applications in transportation, sensor networks, robotics, and power systems. His research was recognized with the 1000-talents award, the IBM High Value Patent award, two IBM invention achievement awards, and the Gerondelis foundation award. Previously, Dr. Freris was Assistant Professor of Electrical and Computer Engineering at New York University Abu Dhabi, and Global Network Assistant Professor of Computer Science at NYU Tandon School of Engineering. Dr. Freris is a senior member of IEEE, and a member of ACM and SIAM.
 \end{IEEEbiography}

 \begin{IEEEbiography}[{\includegraphics[width=1in,height=1.25in,clip,keepaspectratio]{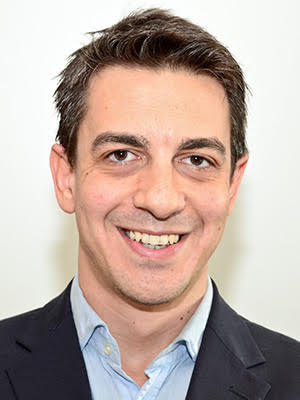}}]{Panagiotis (Panos) Patrinos}
	 is assistant professor at the Department of Electrical Engineering (ESAT) of KU Leuven, Belgium since 2015. During fall/winter 2014 he held a visiting assistant professor position at Stanford University. He received his PhD in Control and Optimization, M.S. in Applied Mathematics and M.Eng. from National Technical University of Athens, in 2010, 2005 and 2003, respectively. After his PhD he held postdoc positions at the University of Trento and IMT Lucca, Italy, where he became an assistant professor in 2012. His current research interests are in the theory and algorithms of structured convex and nonconvex optimization and predictive control with a focus on large-scale, distributed, stochastic and embedded optimization and a wide range of application areas including smart grids, water networks, automotive, aerospace, machine learning and signal processing. 
 \end{IEEEbiography}

\end{document}